    \documentclass[a4paper]{amsart}
    
        \usepackage{latexsym}
        \usepackage{amssymb}
        \usepackage{amsmath}
        \usepackage{amsfonts}
        \usepackage{amsthm}
        \usepackage{mathtools}
        \usepackage{eucal}
        \usepackage{fancyhdr}
        \usepackage[hypertexnames=false]{hyperref}
        \usepackage[all,knot,poly,2cell]{xy}
        \usepackage{fixltx2e}[2005/12/01]

        \newdir{+}{{}*!/-5pt/{}}
        \newdir{>+}{@{>}*!/-5pt/{}}

        \newcommand{\xydblarrow}[2]{\ar@{+->+}@<.3ex>[r]^{\smash[t]{#1}}
          \ar@{+->+}@<-.3ex>[r]_*{\smash[b]{\raisebox{-0.5\height}{$\scriptstyle #2$}}}}

        \newcommand{\mathscr}[1]{\mathcal{{#1}}}
\usepackage{times}

        \newcommand{\llbracket}{[\![} 
       \newcommand{\rrbracket}{]\!]}

        \newcommand{\tensor}{\otimes}
        \newcommand{\homotopic}{\simeq}

        \newcommand{\Z}{\mathbb{Z}}
        \newcommand{\Coll}{\mathscr{C}}
        \newcommand{\Func}[1]{\mathscr{F}_{#1}}
        \renewcommand{\bar}{\overline}
        \DeclareMathOperator{\supp}{supp}
                
        \newcommand{\Orb}{\mathscr{O}}
        \DeclareMathOperator{\Hom}{Hom}
        \DeclareMathOperator{\Ext}{Ext}

        \newcommand{\Identity}[1]{\mathrm{id}_{#1}}
        \DeclareMathOperator{\im}{im}
       
        \DeclareMathOperator{\coker}{coker}
       \DeclareMathOperator{\spf}{Spf}

        \newcommand{\Aff}{\mathbb{A}}
        \DeclareMathOperator{\Pic}{Pic}
       \DeclareMathOperator{\spec}{Spec}
        \DeclareMathOperator{\Isom}{Isom}

  \newcommand{\Et}[1]{\mathbf{Et}({#1})}
  \newcommand{\Etsep}[1]{\mathbf{Et}_{\mathrm{s}}({#1})}
  \newcommand{\Of}[1]{\mathbf{OC}({#1})}
  \newcommand{\Etad}[2]{\mathbf{Et}_{\mathrm{t}/{#2}}({#1})}
  \newcommand{\Etadsep}[2]{\mathbf{Et}_{\mathrm{s},\mathrm{t}/{#2}}({#1})}
  \newcommand{\et}[1]{{#1}_{\mathrm{\acute{e}t}}}
\DeclareMathOperator{\obj}{\mathbf{Obj}}

\renewcommand{\Pr}{\mathbb{P}}

        \newcommand{\itemref}[1]{\eqref{#1}}

        \theoremstyle{plain}
        \newtheorem{thm}{Theorem}[section]
        \newtheorem{cor}[thm]{Corollary}
        \newtheorem{lem}[thm]{Lemma}
        \newtheorem{prop}[thm]{Proposition}

        \newtheorem{mainthms}{Theorem}

        \theoremstyle{definition}
        \newtheorem{defn}[thm]{Definition}
        \newtheorem{ex}[thm]{Example}

        \theoremstyle{remark}
        \newtheorem{rem}[thm]{Remark}

        \numberwithin{equation}{section}

\newcommand{\SEC}[2]{\underline{\mathrm{Sec}}_{{#1}}({#2})} 
\newcommand{\Sec}[2]{\mathrm{Sec}_{{#1}}({#2})}
\newcommand{\FMLSCHABS}{\mathbf{FSch}}
\newcommand{\ACHABS}{\mathbf{AlgStk}}
\newcommand{\SCHABS}{\mathbf{Sch}}
\newcommand{\COH}[1]{\mathbf{Coh}\,({#1})}
\newcommand{\COHP}[2]{\mathbf{Coh}_{\mathrm{p}/{#2}}({#1})}
\newcommand{\HS}[1]{\underline{\mathrm{HS}}_{{#1}}}
\newcommand{\SCH}[1]{\SCHABS/{#1}}
\newcommand{\st}[1]{\mathrm{St}_{{#1}}}
\newcommand{\ACH}[1]{\ACHABS/{#1}}
\newcommand{\qfs}[1]{\mathbf{QF}_{\mathrm{s}}({#1})}
\newcommand{\qfsfin}[1]{\mathbf{QF}_{\mathrm{s}}^{\mathrm{fin}}({#1})}
\newcommand{\RSCH}[1]{\mathbf{RSch}({#1})}
\newcommand{\RAFF}[1]{\mathbf{RAff}({#1})}
\newcommand{\lqfs}[1]{\mathbf{LQF}_{\mathrm{s}}({#1})}
\newcommand{\lqps}[1]{\mathbf{LQP}_{\mathrm{s}}({#1})}
\newcommand{\qfp}[2]{\mathbf{QF}_{\mathrm{s,qp}/{#2}}({#1})}
\newcommand{\qfc}[2]{\mathbf{QF}_{\mathrm{p}/{#2}}({#1})}
\newcommand{\lqf}[2]{\mathbf{LQF}({#1})}
\newcommand{\lqfc}[2]{\mathbf{LQF}_{\mathrm{p}/{#2}}({#1})}

\newcommand{\BETSITE}[1]{({#1})_{\mathrm{\acute{E}t}}}
\newcommand{\SH}[1]{\mathbf{Sh}({#1})}
\newcommand{\HILB}[1]{\underline{\mathrm{Hilb}}_{#1}} 
\newcommand{\HOM}{\mathrm{HOM}}
\newcommand{\FMLSCH}[1]{\FMLSCHABS/{#1}}
\newcommand{\fml}[1]{\mathfrak{{#1}}}
\newcommand{\cmpl}[1]{\widehat{{#1}}}
\newcommand{\HSM}[1]{\HS{{#1}}^{\mathrm{mono}}}

    \title{The Hilbert Stack}
    \author[J. Hall]{Jack Hall}
    \address{Department of Mathematics\\KTH Royal Institute of
      Technology\\SE-100 44 Stockholm\\Sweden}
    \email{jackhall@math.kth.se}
    \author[D. Rydh]{David Rydh}
    \address{Department of Mathematics\\KTH Royal Institute of
      Technology\\SE-100 44 Stockholm\\Sweden}
    \email{dary@math.kth.se}
    \thanks{We would like to 
      sincerely thank Brian Conrad, Jacob Lurie, Martin Olsson, Jason
      Starr, and Ravi Vakil for their comments and suggestions. We would also like to express our gratitude to the referee for their careful reading and excellent suggestions.}
    \date{2013-06-20}

\begin{document}
\begin{abstract}
    Let $\pi \colon X \to S$ be a morphism of algebraic stacks that is locally of finite presentation with affine stabilizers. We prove that there is an algebraic $S$-stack---the Hilbert stack---parameterizing proper algebraic stacks mapping quasi-finitely to $X$. This was previously unknown, even for a morphism of schemes.
\end{abstract}
\maketitle


\section*{Introduction}\label{sec:intro_hs}
Let $\pi \colon X \to S$ be a morphism of algebraic stacks.
Define the \textbf{Hilbert stack}, $\HS{X/S}$, to be
the $S$-stack that sends an $S$-scheme $T$ to the groupoid of
quasi-finite and representable morphisms $(Z \xrightarrow{s} X\times_S
T)$, such that the composition $Z \xrightarrow{s} X\times_S T
\xrightarrow{\pi_T} T$ is proper,  flat, and of finite presentation.

Let $\HSM{X/S}\subset\HS{X/S}$ be the $S$-substack whose objects are those $(Z \xrightarrow{s} X\times_S T)$ such that $s$ is a monomorphism. The
main results of this paper are as follows. 
\begin{mainthms}\label{thm:nonex_h}
  Let $\pi \colon X \to S$ be a \emph{non-separated} morphism of noetherian algebraic stacks. 
  Then $\HSM{X/S}$ is \emph{never} an algebraic stack.  
\end{mainthms}
\begin{mainthms}\label{thm:main}
  Let $\pi\colon X \to S$ be a morphism of algebraic stacks that is locally of finite presentation, with quasi-compact and separated diagonal, and affine stabilizers. 
  Then $\HS{X/S}$ is an algebraic stack, locally of finite presentation over $S$, with quasi-affine diagonal over~$S$.  
\end{mainthms}
\begin{mainthms}\label{thm:main2}
  Let $X \to S$ be a morphism of algebraic stacks that is locally of finite presentation, with quasi-finite and separated diagonal.
  Let $Z \to S$ be a morphism of algebraic stacks that is proper, flat, and of finite presentation with finite diagonal.
  Then the  $S$-stack $T \mapsto \HOM_T(Z\times_S T,X\times_S T)$ is  algebraic, locally of finite presentation over $S$, with quasi-affine  diagonal over $S$.
\end{mainthms}
Let $f \colon Y\to Z$ and $p \colon Z \to W$ be morphisms of stacks. 
Define the fibered category $p_*Y$, the \textbf{restriction of scalars of $Y$ along $p$}, by $(p_*Y)(T) = Y(T\times_W Z)$. 
\begin{mainthms}\label{thm:rest}
  Let $f \colon Y \to Z$ and $p \colon Z \to W$ be morphisms of algebraic
  stacks. Assume that $p$ is proper, flat, and of finite presentation with
  finite diagonal and that $f$ is locally of finite presentation with
  quasi-finite and separated diagonal. Then the restriction of scalars $p_* Y$
  is an algebraic stack, locally of finite presentation over
  $W$, with quasi-affine diagonal over $W$.
\end{mainthms}
Theorem \ref{thm:nonex_h} is similar to the main conclusion of
\cite{MR2369042}, and is included for completeness. In the
case that the morphism $\pi \colon X \to S$ is \emph{separated}, the
Hilbert stack, $\HS{X/S}$, is equivalent to the stack of properly
supported algebras on $X$, which was shown to be algebraic in
\cite{MR2233719}. Thus the new content of this paper is in the
removal of separatedness assumptions from similar theorems in the
existing literature. The statement of Theorem 
\ref{thm:main} for algebraic spaces appeared in
\cite[Appendix~1]{MR0399094}, but was left unproved due to a lack
of foundational results. It is
important to note that Theorems \ref{thm:main}, \ref{thm:main2}, and
\ref{thm:rest} are completely new, even for schemes and algebraic
spaces. 

We wish to point out that if $X$ is an algebraic $S$-stack
with affine stabilizer groups, and $X^0 \subset X$
denotes the open locus where the inertia stack $I_{X/S}$ is
quasi-finite, then there is an isomorphism of $S$-stacks $\HS{X_0/S}
\to \HS{X/S}$. In particular, Theorem \ref{thm:main} is really about algebraic stacks with quasi-finite diagonals. 

Theorems \ref{thm:main2} and \ref{thm:rest} generalize \cite[Thm.\
1.1 \& 1.5]{MR2239345} and \cite{MR2194377,MR2258535} to the non-separated
setting, and follow easily from Theorem
\ref{thm:main}. In the case where $\pi$ is not
flat, as was 
remarked in \cite{hallj_openness_coh}, Artin's Criterion is
difficult to apply. Thus, to 
prove Theorem \ref{thm:main} we use the algebraicity criterion
[\emph{op.\ cit.}, Thm.\ A]. The results of [\emph{op.\ cit.}, \S9]
show that it is sufficient to 
understand how infinitesimal deformations can be extended to global
deformations (i.e.\ the effectivity of formal deformations).    

The difficulty in extending infinitesimal deformations of the Hilbert
stack lies in the dearth of ``formal GAGA'' type results---in the spirit of 
\cite[III.5]{EGA}---for \emph{non-separated} schemes,
algebraic spaces, and algebraic stacks. In this paper, we will prove a
generalization of formal GAGA to 
non-separated morphisms of algebraic 
stacks. The proof of our version of non-separated formal GAGA requires
the development of a number of foundational results on non-separated 
spaces, and forms the bulk of the paper. 
\subsection{Background}
The most fundamental moduli problem in algebraic geometry is the Hilbert
moduli problem for $\Pr^N_{\Z}$: find a scheme that parameterizes
flat families of closed subschemes of $\Pr^N_{\Z}$. It was proven by
Grothendieck \cite[IV.3.1]{MR0146040} that this moduli problem has a
solution which is a disjoint union of projective schemes.

In general, given a morphism of schemes $X \to S$, one may consider
the \textbf{Hilbert moduli problem}: find a scheme $\HILB{X/S}$ 
parameterizing flat families of \emph{closed} subschemes of $X$. It is
more precisely described by its functor of points: for any scheme 
$T$, a map of schemes $T \to \HILB{X/S}$ is equivalent to a diagram: 
\[
\xymatrix{Z  \ar@{^(->}[r] \ar[dr] & \ar[d] X\times_S T\\ & T,}
\]
where the morphism $Z \to X\times_S T$ is a closed immersion, and the
composition $Z \to T$ is proper, flat, and of finite
presentation. Grothendieck, using projective methods, constructed the
scheme $\HILB{\Pr^N_\Z/\Z}$. 

In \cite{MR0260746}, M. Artin developed a new approach to constructing
moduli spaces. It was proved, by M. Artin in \cite[Cor.\
6.2]{MR0260746} and \cite[Appendix]{MR0399094}, that the functor 
$\HILB{X/S}$ had the structure of 
an algebraic space for any \emph{separated} and 
locally finitely presented morphism of algebraic spaces $X \to S$. The
algebraic space  $\HILB{X/S}$ is not, in general, a scheme---even if
$X \to S$ is a proper morphism of smooth complex varieties. In more
recent work, Olsson--Starr \cite{MR2007396} and Olsson \cite{MR2183251}
showed that the functor $\HILB{X/S}$ is an algebraic space in the case
of a separated and locally finitely presented morphism of algebraic
stacks $X \to S$.  

A separatedness assumption on a scheme is rarely restrictive to an
algebraic geometer. Indeed, most schemes algebraic geometers are
interested in are quasi-projective or proper. Let us examine some
spaces that arise in the theory of moduli.
\begin{ex}[Picard Schemes]
  Let $C \to  \Aff^1$ be the family of curves corresponding to a
  conic degenerating to a node. Then the Picard scheme
  $\Pic_{C/\Aff^1}$, which parameterizes families of line bundles on 
  $C/\Aff^1$ modulo pullbacks from the base, is
  not separated. This is worked out in detail in \cite[Ex.\ 
  9.4.14]{MR2222646}. 
\end{ex}
\begin{ex}[Curves]
  Let $\mathcal{U}$ be the stack of \emph{all} curves. That is, a
  morphism $T \to \mathcal{U}$ from a scheme $T$, is equivalent to
  a morphism of algebraic spaces $C \to T$ that is proper, flat,
  finitely presented, and with one-dimensional fibers. In particular,
  $\mathcal{U}$ parameterizes all singular curves, which could be
  non-reduced and have many irreducible and connected components. In
  \cite[Appendix B]{smyth-2009}, it was shown that $\mathcal{U}$ is an
  algebraic stack, locally of finite presentation over $\Z$. The stack
  $\mathcal{U}$ is interesting, as Hassett \cite{MR1957831},
  Schubert \cite{MR1106299}, and Smyth \cite{smyth-2009} have
  constructed modular compactifications of $\mathcal{M}_g$, different
  from the classical Deligne--Mumford compactification
  \cite{MR0262240}, that are open substacks of $\mathcal{U}$. The
  algebraic stack $\mathcal{U}$ is not separated.  
\end{ex}
Unlike schemes, the norm for interesting moduli spaces is that they
are non-separated. Indeed, families of interesting geometric objects tend not to
have unique limits. This is precisely the reason why compactifying
moduli spaces is an active, and very difficult, area of  research. 

Lundkvist and Skjelnes showed in \cite{MR2369042} that for a
non-separated morphism of noetherian algebraic spaces $X \to S$, the functor
$\HILB{X/S}$ is \emph{never} an algebraic space. We
will provide an illustrative example of this phenomenon.  
\begin{ex}\label{ex:sl_nonsep_cl}
  Consider the simplest non-separated scheme: the line with the
  doubled origin. Let $\Bbbk$ be a field and set $S=\spec \Bbbk$. Let $X=\Aff^1_s
  \amalg_{\Aff^1_{s=t}-(0)} \Aff^1_t$, which we view as an $S$-scheme. 
  Now, for a $y$-line, $\Aff^1_y$, we have a $D:=\spec \Bbbk\llbracket x
  \rrbracket$-morphism $T_y \colon D\to \Aff^1_y\times_S D$ given by
  \[
  1\tensor x\mapsto x,\quad y\tensor 1 \mapsto x.
  \]
  Thus, we have an induced $D$-morphism:
  \[
  i \colon D \xrightarrow{T_s}
  (\Aff^1_s)\times_S D \to
  X \times_S D.   
  \]
  Now, the fiber $i_n$ of $i$ over $D_n:=\spec \Bbbk\llbracket
  x\rrbracket/(x^{n+1})$ is topologically the inclusion of one of the two
  origins, which is a closed immersion. Note, however, that the map $i$
  is not a closed immersion. The closed immersions $i_n \colon D_n \to
  X\times_S D_n$ induce compatible $S$-morphisms $D_n 
  \to \HILB{X/S}$. If $\HILB{X/S}$ is an algebraic space, then this
  data induces a unique $S$-morphism $D \to \HILB{X/S}$. That is,
  there exists a closed immersion $j \colon Z \to X\times_S D$ whose fiber over
  $D_n$ is $i_n$. One immediately deduces that $j$ is isomorphic to $i$.
  But $i$ is not a closed immersion, thus we have a contradiction. 
\end{ex}
Note that if $X \to S$ is separated, any monomorphism $Z \to X \times_S T$,
such that $Z \to T$ is proper, is \emph{automatically} a closed
immersion. Thus, for a separated morphism $X \to S$, the stack
$\HSM{X/S}$ is equivalent to the Hilbert functor 
$\HILB{X/S}$. In the case that the morphism $X \to S$ is
non-separated, they are different. Note that in Example
\ref{ex:sl_nonsep_cl}, the deformed object was still a monomorphism,
so will not prove Theorem \ref{thm:nonex_h} for the line with the
doubled-origin.  Let us consider another example.   
\begin{ex}
  Consider the line with doubled-origin again, and retain the
  notation and conventions of Example \ref{ex:sl_nonsep_cl}.  Thus, we
  have an induced map over $D$: 
  \begin{small}
    \[
    D\amalg D \xrightarrow{T_s \amalg T_t} (\Aff^1_s \amalg
    \Aff^1_t)\times_S D\to
    X \times_S D.
    \]
  \end{small}
  \vspace{-6pt}
  {\par\noindent}Where $x=0$, this becomes the inclusion of the doubled
  point, which is a closed immersion. Where $x\neq 0$, this becomes
  non-monomorphic.
\end{ex}
The proof of \cite[Thm.\ 2.6]{MR2369042} is based upon
Example~\ref{ex:sl_nonsep_cl}. Arguing similarly, but with the last example,
one readily obtains Theorem
\ref{thm:nonex_h}. Thus, for a non-separated morphism of schemes, the
obstruction to the existence of a Hilbert scheme is that a
monomorphism $Z \hookrightarrow X$ can deform to a
non-monomorphism. So, one is forced to parameterize non-monomorphic \emph{maps} $Z 
\to X$. Such variants of the Hilbert moduli problem have been considered previously in the literature. We now list those variants that the authors are aware of at the time of publication and that are known to be algebraic.
\begin{itemize}
\item Vistoli's Hilbert stack, which parameterizes families of finite and unramified morphisms to a separated stack \cite{MR1138256}. 
\item The stack of coherent algebras on a separated algebraic stack \cite{MR2233719}.
\item The stack of branchvarieties, which parameterizes geometrically reduced
  algebraic stacks mapping finitely to a separated algebraic stack. It has
  proper components when the target stack has projective coarse moduli space or
  admits a proper flat cover by a quasi-projective scheme
  \cite{MR2608190,MR2233719}.
\item There is a proper algebraic
  space parameterizing Cohen--Macaulay curves with fixed Hilbert
  polynomial mapping finitely, and birationally onto its image, to
  projective space \cite{honsen_phd}.
\item The Hilbert stack of \emph{points} for any morphism of algebraic stacks \cite{MR2821738}. 
\end{itemize}
To subsume the variants of the Hilbert moduli problem listed above,
we parameterize the quasi-finite morphisms $Z \to X$. 
\begin{ex}
  Closed immersions, quasi-compact open immersions, quasi-compact
  unramified morphisms, and finite morphisms are all examples of
  quasi-finite morphisms. By Zariski's Main Theorem \cite[IV.18.12.13]{EGA}, any quasi-finite and separated map of schemes 
  $Z \to X$ factors as $Z \to \bar{Z} \to X$ where $Z \to \bar{Z}$ is
  an open immersion and $\bar{Z} \to X$ is finite. 
\end{ex}
We now define the \textbf{Generalized Hilbert moduli problem}: for a
morphism of algebraic stacks $\pi \colon X \to S$, find an
algebraic stack $\HS{X/S}$ such that a map $T \to \HS{X/S}$ is
equivalent to the data of a quasi-finite and representable map $Z \to
X\times_S T$, with the composition $Z\to T$ proper, flat, and of finite 
presentation. This is the fibered category that appears in Theorem
\ref{thm:main}. The main result in this paper, Theorem \ref{thm:main},
is that this stack is algebraic.  
\begin{ex}
  Let $X \to S$ be a separated morphism of algebraic stacks. Then every
  quasi-finite and separated morphism
  $Z \to X$, such that $Z \to S$ proper, is finite. Hence,
  $\HS{X/S}$ is the stack of properly supported coherent algebras on
  $X$. Lieblich \cite{MR2233719} showed that $\HS{X/S}$ is 
  algebraic whenever $X \to S$ is locally of finite
  presentation and separated.
\end{ex}
\subsection{Outline}\label{subsec:outline}
In \S\ref{ch:hilbstk}, we will prove Theorem  \ref{thm:main} using the
algebraicity criterion \cite[Thm.\ A]{hallj_openness_coh}. To apply
this criterion, like Artin's Criterion \cite[Thm.\ 5.3]{MR0399094}, it
is necessary to know that  formally versal deformations of objects in
the Hilbert stack can be  
effectivized. Note that effectivity results for moduli problems related 
to separated  objects usually follow from the formal GAGA results of
\cite[III.5]{EGA}, and the relevant generalizations to algebraic
stacks \cite{MR2007396,MR2183251}. Since 
we are concerned with non-separated objects, no previously published
effectivity result applies.

In \S\ref{ch:exthm}, we prove a generalization of formal
GAGA for non-separated algebraic stacks. This is the main technical
result of this paper. To be precise, let $R$ be a noetherian ring that
is separated and complete for the topology defined by an ideal $I \subseteq R$ (i.e., $R$ is $I$-adic \cite[$0_{\mathrm{I}}$.7.1.9]{EGA}). Set $S= \spec R$ and for each $n\geq 0$ let $S_n = \spec (R/I^{n+1})$.
Now let $\pi \colon X \to
S$ be a morphism of algebraic stacks that is locally of finite type, with quasi-compact and separated
diagonal, and affine stabilizers. For each $n\geq 0$ let 
$\pi_n \colon X_n
\to S_n$ denote the pullback of the map $\pi$ along the closed immersion
$S_n \hookrightarrow S$. Suppose that for each $n\geq 0$, we have
compatible quasi-finite $S_n$-morphisms $s_n \colon Z_n \to X_n$ such that 
the composition $\pi_n \circ s_n \colon Z_n \to S_n$ is proper. We show (Theorem \ref{thm:exthm}) that there exists a unique, quasi-finite
$S$-morphism $s \colon Z \to X$, such that the composition $\pi \circ s \colon Z
\to S$ is proper, and that there are compatible
$X_n$-isomorphisms $Z\times_X X_n \to Z_n$.  

In \S\S\ref{ch:coeq}-\ref{ch:qfs}, we will develop techniques to
prove the afforementioned effectivity result in \S\ref{ch:exthm}. To motivate these
techniques, it is instructive to explain part of
Grothendieck's proof of formal GAGA for properly supported
coherent sheaves \cite[III.5.1.4]{EGA}. 

So, let $R$ be an $I$-adic noetherian ring, let $S=\spec R$, and for each $n\geq 0$ let $S_n = \spec
(R/I^{n+1})$. For a morphism of schemes $f \colon Y \to S$ that is locally
of finite type and \emph{separated}, let 
$f_n \colon Y_n \to S_n$ denote the pullback of the morphism
$f$ along the closed immersion $S_n \hookrightarrow S$. Suppose that
for each $n\geq 0$, we have a coherent $Y_n$-sheaf $\Func{n}$,
properly supported over $S_n$, together with 
isomorphisms $\Func{n+1}|_{Y_n} \cong \Func{n}$.
Then Grothendieck's formal GAGA \cite[III.5.1.4]{EGA} states that there is a unique coherent
$Y$-sheaf $\Func{}$, with support proper over $S$, such that
$\Func{}|_{Y_n} \cong \Func{n}$. 

For various reasons, it is better to think of adic systems of coherent
sheaves $\{\Func{n}\}$ as  a
coherent sheaf $\fml{F}$ on the formal scheme $\cmpl{Y}$ \cite[I.10.11.3]{EGA}. We say that a coherent sheaf $\fml{F}$ on
the formal scheme $\cmpl{Y}$ is \textbf{effectivizable}, if there exists a
coherent sheaf $\Func{}$
on the scheme $Y$ and an isomorphism of coherent sheaves
$\cmpl{\Func{}} \cong \fml{F}$ on the formal scheme $\cmpl{Y}$. The
effectivity problem
is thus recast as: any coherent sheaf $\fml{F}$ on the formal scheme
$\cmpl{Y}$, with support proper over $\cmpl{S}$, is effectivizable. This is proven using the method of
\emph{d\'evissage} on the abelian category of coherent sheaves with
proper support,
$\COHP{Y}{S}$, on
the scheme $Y$. The proof 
consists of the following sequence of observations. 
\begin{enumerate}
\item \label{dev:step_ext} Given coherent sheaves
  $\mathscr{H}$ and $\mathscr{H}'$ on $Y$, with $\mathscr{H}'$
  properly supported over $S$, the natural map of $R$-modules:
  \[
  \Ext^i_{\Orb_Y}(\mathscr{H},\mathscr{H}') \to
  \Ext^i_{\Orb_{\cmpl{Y}}}(\cmpl{\mathscr{H}},\cmpl{\mathscr{H}}')
  \]
  is an isomorphism for all $i\geq 0$. In particular, the ``$i=0$'' statement shows that the functor
  $\COHP{Y}{S} \to \COHP{\cmpl{Y}}{S}$ is fully faithful.
\item \label{dev:step_dev_qc} By \itemref{dev:step_ext}, it remains to
  prove that the functor $\COHP{Y}{S} \to \COHP{\cmpl{Y}}{S}$ is
  essentially surjective. It is sufficient to prove this essential
  surjectivity when $Y$ is quasi-compact. By noetherian induction on
  $Y$, we may, in addition, assume that if $\fml{F} \in
  \COHP{\cmpl{Y}}{S}$ is annihilated by some coherent ideal
  $J\subseteq \Orb_Y$ with $|\supp(\Orb_Y/J)| \subsetneq |Y|$, then
  $\fml{F}$ is effectivizable.
\item \label{dev:step_2o3} If we have an exact sequence of
  $\cmpl{S}$-properly supported coherent sheaves on 
  $\cmpl{Y}$:
  \[
  \xymatrix{0 \ar[r] & \fml{F}' \ar[r] & \fml{F} \ar[r] & \fml{F}''
    \ar[r] & 0}
  \]
  and two of $\fml{F}'$, $\fml{F}$, $\fml{F}''$ are effectivizable,
  then the third is. This follows from the exactness of completion and
  the $i=0$,$1$ statements of \itemref{dev:step_ext}.
\item \label{dev:step_bir} Combine \itemref{dev:step_dev_qc} and
  \itemref{dev:step_2o3} to deduce that if $\alpha \colon \fml{F} \to
  \fml{F}'$ is a morphism in $\COHP{\cmpl{Y}}{S}$  such that
  $\fml{F}'$ is effectivizable and $\ker \alpha$ and $\coker
  \alpha$ are annihilated by some  coherent ideal
  $J\subseteq \Orb_Y$ with $|\supp(\Orb_Y/J)| \subsetneq |Y|$, then
  $\fml{F}$ is effectivizable.
\item \label{dev:step_proj} The result is true for quasi-projective morphisms
  $Y \to S$. This is proved via a direct argument. 
\item \label{dev:step_chow} The Chow Lemma \cite[II.5.6.1]{EGA} gives
  a quasi-projective $S$-scheme $Y'$ and a projective $S$-morphism $p\colon Y' \to Y$ that is an isomorphism over a dense open subset $U$ of $Y$.
\item \label{dev:step_pb} Use  \itemref{dev:step_proj} for the
  quasi-projective morphism $Y'\to S$ to 
  show that $\cmpl{p}^*\fml{F} \cong \cmpl{\mathscr{G}}$ for some 
  $S$-properly supported coherent sheaf $\mathscr{G}$ on $Y'$.
\item \label{dev:step_ff} Use the Theorem on Formal Functions
  \cite[III.4.1.5]{EGA} to show that 
  $(p_*\mathscr{G})^\wedge \cong \cmpl{p}_*\cmpl{p}^*\fml{F}$. Thus
  $\cmpl{p}_*\cmpl{p}^*\fml{F}$ is effectivizable with $\cmpl{S}$-proper
  support. 
\item \label{dev:step_adj} Note that we have an adjunction morphism $\eta
  \colon \fml{F} \to \cmpl{p}_*\cmpl{p}^*\fml{F}$. Pick a coherent ideal $J
  \subseteq \Orb_Y$ defining the complement of $U$. By
  \cite[III.5.3.4]{EGA}, we may choose $J$ so that it annihilates the
  kernel and cokernel of $\eta$. By \itemref{dev:step_bir}, $\fml{F}$
  is effectivizable, and the result follows. 
\end{enumerate}
The proof of the non-separated effectivity result in \S\ref{ch:exthm} will be
very similar to the technique outlined above, once the steps are
appropriately reinterpreted. For a non-separated morphism of schemes
$X \to S$, instead of the abelian category $\COH{X}$ (resp.~$\COHP{X}{S}$), we consider the
non-abelian category $\qfs{X}$ (resp.~$\qfc{X}{S}$) that consists of quasi-finite and
separated morphisms $Z \to X$ (resp.~quasi-finite and separated
morphism $Z \to X$ such that $Z \to S$ is proper). Instead of extensions,
which do not make sense in a non-abelian category, we will use finite limits.

In 
\S\ref{ch:coeq}, we will show that $\qfs{X}$ is closed under finite
limits along finite morphisms. In \S\ref{ch:coeq} we also
reinterpret \itemref{dev:step_ext} and 
\itemref{dev:step_2o3} in terms of the preservation of these
finite limits under completion. The
analogue of quasi-projective morphisms $Y \to S$ in \itemref{dev:step_proj}
and \itemref{dev:step_pb} are morphisms that factor as $Y \to Y' \to
S$, where $Y \to Y'$ is 
\'etale, and $Y' \to S$ is projective---note that it is essential that
we do not assume that $Y \to Y'$ is separated. The analogue of the Chow Lemma used in \itemref{dev:step_chow},
is a generalization, for non-separated schemes and algebraic spaces,
due to Raynaud--Gruson \cite[Cor.\  
5.7.13]{MR0308104}. In \S\ref{ch:qfs}, for a proper morphism $q \colon Y'
\to Y$,  we construct an adjoint pair $(q_!,q^*) \colon \qfs{Y'} \leftrightarrows
\qfs{Y}$ which takes the role of the adjoint pair $(q^*,q_*) \colon \COH{Y}
\leftrightarrows \COH{Y'}$. We also show that the adjoint pair can be 
constructed for locally noetherian formal schemes, and prove an
analogue of \itemref{dev:step_ff} in this
setting. In \S\ref{ch:exthm}, we will combine the results
of \S\S\ref{ch:coeq} and \ref{ch:qfs} to obtain analogues of
\itemref{dev:step_bir} and \cite[III.5.3.4]{EGA} in order to complete
the analogue of step \itemref{dev:step_adj}.

In \S\ref{ch:stein_factorizations}, we introduce the generalized Stein
factorization: every separated morphism of finite type with proper fibers
factors canonically as a proper Stein morphism followed by a quasi-finite
morphism. The adjoint pair of \S\ref{ch:qfs} is an immediate consequence of the
existence of the generalized Stein factorization.

\subsection{Notation}\label{sec:notn}
We introduce some notation here that will be used throughout the
paper. For a category $\Coll$ and $X \in \obj\Coll$, we have the
\textbf{slice} category $\Coll/X$, with objects the morphisms $V \to
X$ in $\Coll$, and morphisms commuting diagrams over $X$, which are
called $X$-morphisms. If the category $\Coll$ has finite limits and $f \colon Y
\to X$ is a morphism in $\Coll$, then for $(V\to X) \in
\obj(\Coll/X)$, define $V_Y := V \times_X Y$. Given a morphism $p \colon V' 
\to V$ in $\Coll/X$ there is an induced morphism $p_Y \colon V'_Y \to V_Y$
in $\Coll/Y$. There is an
induced functor $f^* \colon \Coll/X \to \Coll/Y$ given by
$(V \to X) \mapsto (V_Y \to Y)$. 

Given a ringed space $U := (|U|,\Orb_U)$, a sheaf of ideals
$\mathscr{I} \subseteq \Orb_U$, and a morphism of ringed spaces $g \colon V
\to U$, we define the \textbf{pulled back ideal} $\mathscr{I}_V = 
\im(g^*\mathscr{I} \to \Orb_V)\subseteq \Orb_V$. 

Let $S$ be a scheme. An algebraic $S$-space is a sheaf $F$ on the big
\'etale site $\BETSITE{\SCH{S}}$, such that the diagonal morphism
$\Delta_F \colon F \to F \times_S F$ is represented by schemes, and there
is a smooth surjection $U \to F$ from an $S$-scheme $U$. An
algebraic $S$-stack is a stack $H$ on $\BETSITE{\SCH{S}}$, such that
the diagonal morphism $\Delta_H \colon H \to H\times_S H$ is represented
by algebraic $S$-spaces and there is a smooth surjection $U \to H$
from an algebraic $S$-space $U$. A priori, we make no separation
assumptions on our algebraic stacks.  We do show, however, that all
algebraic stacks figuring in this paper possess quasi-compact and
separated diagonals. Thus all of the results of \cite{MR1771927}
apply. We denote the $(2,1)$-category of algebraic stacks by $\ACHABS$. 


\section{Finite colimits of quasi-finite morphisms} \label{ch:coeq}
In this section we will prove that the
colimit of a finite diagram of finite morphisms between quasi-finite
objects over an algebraic stack exists
(Theorem~\ref{thm:pushouts}). We will also prove that these colimits are
compatible with completions (Theorem~\ref{thm:fml_pushouts_commute} and
Corollary~\ref{cor:fml_pushouts_commute_stks}).

Let $X$ be an algebraic stack. Denote the $2$-category of algebraic
stacks over $X$ by $\ACH{X}$. Define the full
$2$-subcategory $\RSCH{X}\subset \ACH{X}$ to have those objects $Y
\xrightarrow{s} X$, where the morphism $s$ is schematic. We now make
three simple observations.
\begin{enumerate}
\item The $1$-morphisms in $\RSCH{X}$ are schematic.
\item Automorphisms of $1$-morphisms in $\RSCH{X}$ are trivial. Thus $\RSCH{X}$ is naturally $2$-equivalent to a $1$-category.
\item If the algebraic stack $X$ is a scheme, then the natural functor
  $\SCH{X} \to \RSCH{X}$ is an equivalence of categories.  
\end{enumerate}
\begin{defn}\label{defn:qf}
  Let $X$ be an algebraic stack. Define the following full
  subcategories of $\RSCH{X}$:
  \begin{enumerate}
  \item $\RAFF{X}$ has objects the affine morphisms to $X$; 
  \item $\qfs{X}$ has objects the quasi-finite and separated maps to
    $X$. 
  \end{enumerate}
  We also let $\qfsfin{X}\subseteq \qfs{X}$ denote the subcategory
  that has the same objects as $\qfs{X}$ but only has the morphisms
  $f \colon (Z \xrightarrow{s} X) \to (Z' \xrightarrow{s'} X)$ such that
  $Z\to Z'$ is finite. This is not a full subcategory.
\end{defn}
Recall that quasi-finite, separated, and representable morphisms of
algebraic stacks are schematic~\cite[Thm.~A.2]{MR1771927}. The
subcategories $\qfs{X}$, $\RAFF{X}$ and $\RSCH{X}$ have all finite limits
and these coincide with the limits in $\ACH{X}$. The subcategory $\qfsfin{X}$
has fiber products and these coincide with the fiber products in $\ACH{X}$.
However, there is not a final object in $\qfsfin{X}$,
except when every quasi-finite morphism to $X$ is finite.
\subsection{Algebraic Stacks}\label{sec:coeq_sch}
For the moment, we will be primarily concerned with the existence
of colimits in the category $\qfs{X}$, where $X$ is a locally
noetherian algebraic stack. Colimits in algebraic geometry are usually 
subtle, so we restrict our attention to finite colimits in $\qfsfin{X}$.
This includes the two types that will be useful in
\S\ref{ch:exthm}---pushouts of two finite morphisms and coequalizers of
two finite morphisms. We will, however, pay
attention to some more general 
types of colimits in the $2$-category of algebraic stacks, as the added
flexibility will simplify the exposition.
\begin{rem}
  Note that because we have a fully faithful embedding of categories
  $\qfs{X} \subset \RSCH{X}$, a categorical colimit in $\RSCH{X}$,
  that lies in $\qfs{X}$, is automatically a categorical colimit in
  $\qfs{X}$. This will happen frequently in this section.
  Moreover, if a finite diagram $\{Z_i\}_{i\in I}$ in $\qfsfin{X}$ has a
  categorical colimit $Z$ in $\qfs{X}$ and $Z_i\to Z$ is finite for every $i\in
  I$, then $Z$ is also a categorical colimit in $\qfsfin{X}$. Indeed, it is
  readily seen that $\coprod_{i\in I}Z_i\to Z$ is surjective so any morphism
  $(Z\to W)\in \qfs{X}$, such that $Z_i\to Z\to W$ is finite for every $i\in
  I$, is finite.
\end{rem}
\begin{rem}\label{rem:algstk_col}
  It is important to observe that a categorical colimit in $\ACH{X}$
  is, in general, different from a categorical colimit in
  $\RSCH{X}$. Indeed, let $X = \spec \Bbbk$, and let $G$ be a
  finite group, then the $X$-stack $\mathrm{B}G$ is the colimit in
  $\ACH{X}$ of the diagram $[G_X \rightrightarrows X]$. The colimit of
  this diagram in
  the category $\RSCH{X}$ is just $X$. 
\end{rem}
The definitions that follow are closely related to those given in
\cite[Def.~2.2]{Rydh:2007p231}. Recall that a map of topological spaces $g
\colon U \to V$ is \textbf{submersive} if a subset $Z\subseteq |V|$ is open
if and only if $g^{-1}(Z)$ is an open subset of $|U|$.
\begin{defn}
  Consider a diagram of algebraic stacks $\{Z_i\}_{i\in I}$, an
  algebraic stack $Z$, and suppose that we have compatible maps
  $\phi_i \colon Z_i \to Z$ for all $i\in I$. Then we say that the data $(Z,\{\phi_i\}_{i\in I})$ is a
  \begin{enumerate}
  \item \textbf{Zariski colimit} if the induced map on topological
    spaces $\phi \colon\varinjlim_i |Z_i| \to |Z|$ is a homeomorphism
    (equivalently, the map $\amalg_{i\in I} \phi_i \colon \amalg_{i\in I}
    Z_i \to Z$ is submersive and the map $\phi$ is a bijection of sets);  
  \item \textbf{weak geometric colimit} if it is a Zariski colimit of
    the diagram $\{Z_i\}_{i\in I}$, and the canonical map of
    lisse-\'etale sheaves of rings $\Orb_Z \to
    \varprojlim_i (\phi_i)_*\Orb_{Z_i}$ 
    is an isomorphism;   
  \item \textbf{universal Zariski colimit}  if for any algebraic
    $Z$-stack $Y$,  the data $(Y, \{(\phi_i)_Y\}_{i\in I})$ is
    a Zariski colimit of the diagram $\{(Z_i)_Y\}_{i\in I}$;
  \item \textbf{geometric colimit} if it is a
    universal Zariski and weak geometric colimit of the diagram
    $\{Z_i\}_{i\in I}$;
  \item \textbf{uniform geometric colimit} if for any \emph{flat},
    algebraic $Z$-stack $Y$, the data $(Y,\{(\phi_i)_Y\}_{i\in
      I})$ is a geometric colimit of the diagram
    $\{(Z_i)_Y\}_{i\in I}$. 
  \end{enumerate}
\end{defn}
The exactness of flat pullback of sheaves and flat base change easily proves
\begin{lem}\label{lem:fin_loc_col}
  Suppose that we have a \emph{finite} diagram of algebraic stacks
  $\{Z_i\}_{i\in I}$, then a geometric colimit $(Z,\{\phi_i\}_{i\in
    I})$ is a uniform geometric colimit. 
\end{lem}
Also, note that weak geometric colimits in the category of
algebraic stacks are not unique and thus are not categorical colimits (Remark \ref{rem:algstk_col}). In the
setting of schemes, however, we have the following Lemma.
\begin{lem}\label{lem:sch_wgeom_colim}
Let $\{Z_i\}_{i\in I}$ be a diagram of schemes.
Then every weak geometric colimit $(Z,\{\phi_i\}_{i\in I})$ is a colimit in
the category of locally ringed spaces and, consequently, a colimit in the
category of schemes. 
\end{lem}
\begin{proof}
  Fix a \emph{ringed} space
  $W=(|W|,\Orb_W)$, together with compatible 
  morphisms of ringed spaces $\psi_i \colon Z_i \to W$. We will 
  produce a unique map of ringed spaces $\psi \colon Z 
  \to W$ satisfying $\psi\phi_i = \psi_i$ for all $i\in I$. Since
  $(Z,\{\phi_i\}_{i\in I})$ is a Zariski colimit, it follows that
  there exists a unique, continuous, map of topological spaces $\psi \colon |Z|
  \to |W|$ such that $\psi\phi_i = \psi_i$. On the lisse-\'etale site,
  we also know that the map $\Orb_Z \to \varprojlim_i
  (\phi_i)_*\Orb_{Z_i}$ is an isomorphism. The natural functor from
  the lisse-\'etale site of $Z$ to the Zariski site of $Z$ also
  preserves limits.  In particular, we deduce that as sheaves of rings
  on the topological space $|Z|$, the natural map $\Orb_Z \to
  \varprojlim_i (\phi_i)_*\Orb_{Z_i}$ is an isomorphism. The functor
  $\psi_*$ preserves limits, thus the maps $\Orb_{W} \to
  (\psi_i)_*\Orb_{Z_i}$ induce a unique map to
  $\psi_*\Orb_Z$. Hence, we obtain a unique morphism of ringed spaces
  $\psi \colon Z \to W$ that is compatible with the data. 
  
  To complete the proof it remains to show that if, in addition,
  $W=(|W|,\Orb_W)$ is a \emph{locally} ringed space and the morphisms
  $\psi_i$ are morphisms of locally ringed spaces, then $\psi$ is a
  morphism of locally ringed spaces. Let $z\in |Z|$. We must
  prove that the induced morphism $\Orb_{W,\psi(z)} \to \Orb_{Z,z}$ is
  local. Since $\amalg_{i\in I} |Z_i| \to |Z|$ is surjective, there
  exists an $i\in I$ and $z_i \in |Z_i|$ such that $\phi_i(z_i) =
  z$. By hypothesis, the morphism $\Orb_{W,\psi(z)} \to \Orb_{Z_i,z_i}$
  is local. Also, $Z_i \to Z$ is a morphism of schemes so
  $\Orb_{Z,z} \to \Orb_{Z_i,z_i}$ is also local. We immediately deduce
  that the morphism $\Orb_{W,\psi(z)} \to \Orb_{Z,z}$ is local, and the
  result is proven.
\end{proof}
The following criterion will be useful for verifying when a colimit is
a universal Zariski colimit.
\begin{lem}\label{lem:u_zar_col}
  Consider a diagram of algebraic stacks $\{Z_i\}_{i\in I}$, an
  algebraic stack $Z$, and suppose that we have compatible maps $\phi_i \colon Z_i \to Z$ for all $i\in I$. If the map
  $\amalg_{i} \phi_i \colon \amalg_i Z_i \to Z$ is surjective and
  universally submersive and also
  \begin{enumerate}
  \item \label{lem:u_zar_col1} for any geometric point $\spec K \to
    Z$, the map $\phi_K \colon 
    \varinjlim_i |(Z_i)_K| \to |\spec K|$ is an injection of sets; or 
  \item \label{lem:u_zar_col2} there is an algebraic stack $X$ such
    that $Z, Z_i \in \qfs{X}$ for all $i\in I$, the maps $\phi_i$
    are $X$-maps, and for any geometric point $\spec L \to X$, the map
    $\phi_L \colon \varinjlim_i |(Z_i)_L| \to |Z_L|$ is an injection of sets,
  \end{enumerate}
  then $(Z,\{\phi_i\}_{i\in I})$ is a universal Zariski colimit of the
  diagram $\{Z_i\}_{i\in I}$.
\end{lem}
\begin{proof}
  To show \itemref{lem:u_zar_col1}, we observe that the universal
  submersiveness hypothesis on the map $\amalg_i \phi_i \colon \amalg_i Z_i
  \to Z$ reduces the statement to showing that for any morphism of
  algebraic stacks $Y \to Z$, the map $\phi_Y \colon \varinjlim_i
  |(Z_i)_Y| \to |Y|$ is a bijection of sets, which will follow if the
  map $\phi_K \colon \varinjlim_i |(Z_i)_K| \to |\spec K|$ is bijective for
  any geometric point $\spec K \to Y$. By assumption, we know that
  $\phi_K$ is injective.  For the surjectivity, we note that we have a
  commutative diagram of sets:
  \[
  \xymatrix{\coprod_i |(Z_i)_K| \ar[r]^{\alpha_K} \ar[dr]_{\amalg_i \phi_i} &
    \varinjlim_i |(Z_i)_K| \ar[d]^{\phi_K} \\ & |\spec K| }, 
  \]
  where $\alpha_K$ and $\amalg_i \phi_i$ are surjective, thus $\phi_K$
  is also surjective. For  \itemref{lem:u_zar_col2}, given a geometric
  point $\spec L \to X$, then $Z(\spec L) = |Z_L|$ and
  $Z_i(\spec L) = |(Z_i)_L|$, thus we may apply the
  criterion of \itemref{lem:u_zar_col1} to obtain the claim.
\end{proof}
To obtain similarly useful results for stacks, we will need some relative
notions, unlike the previous definitions which were all absolute. 
\begin{defn}
  Fix an algebraic stack $X$ and a diagram $\{Z_i\}_{i\in I}$
  in $\RSCH{X}$. Suppose that $Z\in
  \RSCH{X}$, and that we have compatible $X$-morphisms
  $\phi_i \colon Z_i \to Z$ for all $i\in I$. We say that the data
  $(Z,\{\phi_i\}_{i\in I})$ is a \textbf{uniform categorical colimit
    in $\RSCH{X}$} if for any \emph{flat}, algebraic
  $X$-stack $Y$, the data $(Z_Y,\{(\phi_i)_Y\}_{i\in I})$ is the categorical
  colimit of the diagram
  $\{(Z_i)_Y\}_{i\in I}$ in $\RSCH{Y}$.
\end{defn}
Combining Lemmas \ref{lem:fin_loc_col} and \ref{lem:sch_wgeom_colim}
we obtain  
\begin{prop}\label{prop:fin_col_stk}
  Let $X$ be an algebraic stack, and suppose that $\{Z_i\}_{i\in I}$
  is a finite diagram in $\RSCH{X}$. Then a 
  geometric colimit $(Z,\{\phi_i\}_{i\in I})$ in $\RSCH{X}$ is 
  also a uniform geometric and uniform categorical colimit in
  $\RSCH{X}$.  
\end{prop}
\begin{proof}
  By Lemma \ref{lem:fin_loc_col}, it suffices to show that $Z$ is a
  categorical colimit. Let $U \to X$ be a smooth surjection from a
  scheme. By Lemmas \ref{lem:fin_loc_col} and
  \ref{lem:sch_wgeom_colim}, we see that $(Z_U, \{(\phi_i)_U\}_{i\in
    I})$ is a categorical colimit of the diagram
  $\{(Z_i)_U\}_{i\in I}$ in $\RSCH{U}$, and remains so after flat schematic
  base change on $U$. Suppose that we have compatible 
  $X$-morphisms 
  $\alpha_i \colon Z_i \to W$. Then we want to show that there is a unique
  $X$-morphism $\alpha\colon Z \to W$ that is compatible with this
  data. Let $R = U\times_X U$, and let
  $s$, $t \colon R \to U$ denote the two projections. Since $Z_U$ is
  the categorical colimit in $\RSCH{U}$, there is a unique 
  $X$-morphism $\beta_U \colon Z_U \to W_U$ that is
  compatible with $(\alpha_i)_U \colon (Z_i)_U \to W_U$. 

  Suppose for the
  moment that $X$ is an algebraic space. Then $R$ 
  is a scheme and the maps $s$, $t \colon R \to U$ are morphisms of
  schemes.  In particular, we know that $Z_R$ is a
  categorical colimit in $\RSCH{R}$ and so there is a unique morphism
  $\beta_R \colon Z_R \to W_R$ that is compatible with the
  morphisms $(\alpha_i)_R \colon (Z_i)_R \to W_R$. Noting
  that $(\beta_U)\times_{U,s} R$ and $(\beta_U)\times_{U,t} R$ are also
  such maps, we conclude that these maps are actually equal (as maps
  of algebraic spaces) and so by smooth descent, we conclude that
  there is a unique $X$-morphism $\alpha \colon Z \to W$, and so
  $(Z,\{\phi_i\}_{i\in I})$ is 
  a categorical colimit if $X$ is an algebraic space. Applying Lemma
  \ref{lem:fin_loc_col}, one is able to conclude that
  $(Z,\{\phi_i\}_{i\in I})$ remains a categorical colimit after flat
  \emph{representable} base change on $X$, for all algebraic spaces
  $X$. 

  Now, returning to the case that $X$ is an algebraic stack, we know
  that $s$, $t \colon R \to U$ are flat representable morphisms, repeating
  the descent argument given above shows that $(Z,\{\phi_i\})_{i\in
    I}$ is a categorical colimit in $\RSCH{X}$.  
\end{proof}
The following result is a generalization of \cite[Lem.\ 
17]{kollar-2008}, which treats the case of a finite equivalence
relation of schemes.
\begin{thm}\label{thm:pushouts}
  Let $X$ be a locally noetherian algebraic stack. Let 
  $\{Z_i\}_{i\in I}$ be a finite diagram in $\qfsfin{X}$.
  Then a colimit $(Z,\{\phi_i\}_{i\in I})$ exists in $\qfsfin{X}$. Moreover,
  this colimit is a uniform categorical colimit 
  in $\RSCH{X}$ and a uniform geometric colimit.
\end{thm}
We will need some lemmas to prove Theorem \ref{thm:pushouts}.
We first treat the affine case, then the finite case and finally the
quasi-finite case.
\begin{lem}\label{lem:aff_pushouts}
  Let $X$ be an algebraic stack. Let $\{Z_i\}_{i\in I}$
  be a finite diagram in $\RAFF{X}$. Then this diagram has a
  categorical colimit in $\RAFF{X}$, whose formation commutes with
  flat base change on $X$.
\end{lem}
\begin{proof}
  By \cite[Prop.\ 14.2.4]{MR1771927}, there is an 
  anti-equivalence of categories between $\RAFF{X}$ and the category
  of quasi-coherent sheaves of $\Orb_X$-algebras, which commutes with 
  arbitrary change of base, and is given by $(Z\xrightarrow{s} X)
  \mapsto (\Orb_X \xrightarrow{\smash[t]{s^\sharp}} s_*\Orb_Z)$.  Since the
  category of quasi-coherent $\Orb_X$-algebras has finite limits, it
  follows that if $s_i \colon Z_i \to X$ denotes the structure map of $Z_i$,
  then the categorical 
  colimit is $Z = \spec_X \varprojlim_i (s_i)_*\Orb_{Z_i}$. Since
  flat pullback of sheaves is exact, the formation of this colimit
  commutes with flat base change on $X$.
\end{proof}
\begin{lem}\label{lem:fin_pushouts}
  Let $X$ be a locally noetherian algebraic stack. Let
  $\{Z_i\}_{i\in I}$ be a finite diagram in
  $\RSCH{X}$ such that $Z_i\to X$ is finite for every $i\in I$.
  Then this diagram has a uniform
  categorical and uniform geometric colimit $Z$ in $\RSCH{X}$ which
  is finite over $X$.
\end{lem}
\begin{proof}
  By Lemma \ref{lem:aff_pushouts} the diagram has a 
  categorical colimit $(Z,\{\phi_i\}_{i\in I})$ in $\RAFF{X}$. 
  If $s_i \colon Z_i \to X$ and $s\colon Z\to X$ denote the structure maps, then
  $s_*\Orb_Z\subseteq \prod (s_i)_* \Orb_{Z_i}$.
  Since $X$ is locally
  noetherian, it follows that $Z$ is finite over $X$. By
  Proposition \ref{prop:fin_col_stk}, it remains to show that $Z$ is a
  geometric colimit. Note that since $\coprod_{i\in I} Z_i\to Z$
  is dominant and finite, it is surjective, universally
  closed and thus universally submersive.

  We will now show that $Z$
  is a universal Zariski colimit using the criterion of Lemma
  \ref{lem:u_zar_col}\itemref{lem:u_zar_col2}. To apply this
  criterion we will employ a generalization of the arguments given in \cite[Lem.\ 17]{kollar-2008}.
  Let $\bar{x} \colon \spec \Bbbk
  \to X$ be a geometric point. By
  \cite[0\textsubscript{III}.10.3.1]{EGA} this map factors as
  $\spec \Bbbk 
  \xrightarrow{\smash[t]{\bar{x}^1}} X^1 \xrightarrow{p} X$, where $p$ is flat and
  $X^1$ is the spectrum of a maximal-adically complete, local noetherian
  ring with residue field $\Bbbk$. Since the algebraic stack $Z$ is
  a uniform categorical colimit in $\RAFF{X}$, 
  we may replace $X$ by $X^1$, and we denote the unique closed point of
  $X$ by $x$. 

  For a finite $X$-scheme $U$, let 
  $\pi_0(U)$ be its set of connected components. The scheme $X$ is
  henselian \cite[IV.18.5.14]{EGA}, so there is a unique universal
  homeomorphism $h_{U} \colon U_x \to 
  \amalg_{m\in \pi_0(U)} \{x\}$ which is functorial with respect to
  $U$. In particular, there is a unique factorization $U
  \xrightarrow{s_U} \amalg_{m\in \pi_0(U)} X \to X$ such that
  $(s_U)_x = h_U$.

  Let $W_i=\coprod_{m\in \pi_0(Z_i)} X$. Since $\pi_0(-)$ is a functor, we
  obtain a diagram $\{W_i\}_{i\in I}$ in $\RAFF{X}$ and we let $W$ be the
  categorical colimit of this diagram in $\RAFF{X}$. It is readily seen that
  $W=\coprod_{m\in \varinjlim \pi_0(Z_i)} X$ so that there is a bijection of
  sets $\varinjlim_i \pi_0(Z_i)\to \pi_0(W)$.  Since $Z$ is a categorical
  colimit, there is a canonical map $\mu\colon Z \to W$.
  In particular, the
  bijection $\nu_x \colon \varinjlim_i |(Z_i)_x|
  \to |W_x|$ factors as $\varinjlim_i |(Z_i)_x|
  \xrightarrow{\psi_x} |Z_x| \xrightarrow{\mu_x} |W_x|$ and
  thus  $\psi_x \colon \varinjlim_i |(Z_i)_x| \to |Z_x|$ is
  injective. Hence, we have shown that  $Z$ is a universal 
  Zariski colimit and it remains to show that $Z$ has the correct
  functions. 
  
  There is a canonical morphism of sheaves of $\Orb_X$-algebras
  $\epsilon\colon\Orb_Z \to \varprojlim_i (\phi_i)_*\Orb_{Z_i}$,
  which we have to show is an isomorphism. By functoriality,
  we have an induced morphism of
  sheaves of $\Orb_X$-algebras:
  \[
  \epsilon_2 \colon s_*\Orb_{Z} \xrightarrow{s_*\epsilon}
    s_*\bigl(\varprojlim_i (\phi_i)_*\Orb_{Z_i}\bigr) \xrightarrow{\epsilon_1}
    \varprojlim_i s_*(\phi_i)_*\Orb_{Z_i}.
  \]
  Since the functor $s_*$ is left exact, $\epsilon_1$ is an
  isomorphism; by construction of $Z$ the map $\epsilon_2$ is an 
  isomorphism and so $s_*\epsilon$ is an isomorphism. Since 
  $s$ is affine, the functor $s_*$ is faithfully exact and we
  conclude that the map $\epsilon$ is an isomorphism of sheaves.
\end{proof}
\begin{proof}[Proof of Theorem \ref{thm:pushouts}]
By hypothesis, the structure morphisms $s_i \colon Z_i \to X$
are quasi-finite, separated, and representable. By 
Zariski's Main Theorem \cite[Thm.\ 16.5(ii)]{MR1771927}, there are finite    
$X$-morphisms $\bar{s_i}\colon \overline{Z_i} \to X$ and open
immersions $u_i \colon Z_i
\hookrightarrow \overline{Z_i}$. Let $W_i=\prod_{i\to j} \overline{Z_j}$
where the product is fibered over $X$. There is an induced morphism
$Z_i\to W_i$ given by the composition
$Z_i\xrightarrow{h_*} Z_j\xrightarrow{u_j} \overline{Z_j}$ over
the factor corresponding to the arrow
$h\colon i\to j$. Using the composition in $I$, there is a natural diagram
$\{W_i\}_{i\in I}$ such that
the morphisms $Z_i\to W_i$ induce a morphism of diagrams.
%
%
Since $Z_i\to \overline{Z_i}$ is an immersion, so is $Z_i\to W_i$.
Now, replace $W_i$ with the schematic image of $Z_i\to W_i$. Then
$Z_i\to W_i$ is an open dense immersion and
$\{W_i\}_{i\in I}$ is a diagram with finite structure morphisms $W_i\to X$.
Thus, the diagram $\{W_i\}_{i\in I}$ has a uniform
geometric and uniform categorical colimit $W\in\RSCH{X}$ and
$W \to X$ is finite (Lemma \ref{lem:fin_pushouts}).

We will now show that for any arrow $h\colon i\to j$, the open substack $Z_j\subset
W_j$ is pulled back to the open substack $Z_i\subset W_i$.
We have canonical maps $Z_i
\xrightarrow{\alpha} Z_j\times_{W_j} W_i \xrightarrow{\beta} W_i$ and since
the maps $\beta\circ \alpha$ and $\beta$ are open immersions, so is $\alpha
\colon Z_i \to Z_j\times_{W_j} W_i$. Similarly, we have $Z_i
\xrightarrow{\alpha} Z_j\times_{W_j} W_i \xrightarrow{\gamma} Z_j$ where
$\gamma\circ \alpha$ and $\gamma$ are finite morphisms.
Thus $\alpha \colon Z_i \to Z_j\times_{W_j} W_i$ is open and
closed. Since $u_i=\beta\circ \alpha \colon Z_i \to W_i$ is dense, we conclude
that $Z_i = Z_j\times_{W_j} W_i$. 

Let $|Z|$ be the set-theoretic image of $\coprod_i |Z_i|$ in
$|W|$. As $W$ is a Zariski colimit, it follows that
$|Z|=\varinjlim_i |Z_i|$ is open in $|W|$. We let $Z\subset
W$ be the open substack with underlying topological space $|Z|$.
Then the diagram $\{Z_i\}_{i\in I}$
is obtained as the pull-back of the diagram $\{W_i\}_{i\in I}$
along the open immersion $Z\to W$. Since $W$ is a uniform geometric
colimit, the pull-back $Z$ is a uniform geometric colimit. As $Z\in\RSCH{X}$,
it is a uniform categorical colimit in $\RSCH{X}$ by
Proposition~\ref{prop:fin_col_stk}.
\end{proof}
\subsection{Completions of schemes}
Here we will show that the colimits constructed in Theorem
\ref{thm:pushouts} remain colimits after completing along a closed
subset. Denote the category of formal schemes by
$\FMLSCHABS$. We require some more definitions that are analogous to
those given in \S\ref{sec:coeq_sch}.  
\begin{defn}
  Consider a diagram of formal schemes $\{\fml{Z}_i\}_{i\in
    I}$, a formal scheme $\fml{Z}$, and suppose that we have
  compatible maps $\varphi_i \colon  \fml{Z}_i \to \fml{Z}$ for 
  every $i\in I$. Then we say that the data
  $(\fml{Z},\{\varphi_i\}_{i\in I})$ is a  
  \begin{enumerate}
  \item \textbf{formal Zariski colimit} if the induced map on
    topological spaces $\varinjlim_i |\fml{Z}_i| \to |\fml{Z}|$ is a
    homeomorphism;
  \item \textbf{formal weak geometric colimit} if
    it is a formal Zariski colimit and the canonical map of sheaves of
    rings $\Orb_{\fml{Z}} \to \varprojlim_i
    (\varphi_i)_*\Orb_{\fml{Z}_i}$ is a topological isomorphism, where
    we give the latter sheaf of rings the limit topology (this is
    nothing other than $\fml{Z}$ being the colimit in the category of
    topologically ringed spaces).
  \end{enumerate}
  If, in addition, the formal scheme $\fml{Z}$ is locally noetherian,
  and the maps $\varphi_i \colon \fml{Z}_i \to \fml{Z}$ are topologically
  of finite type, then we say that the data
  $(\fml{Z},\{\varphi_i\}_{i\in I})$ is a 
  \begin{enumerate}
  \item[(3)] \textbf{universal formal Zariski colimit}  if for any
    adic formal $\fml{Z}$-scheme $\fml{Y}$, the data
    $({\fml{Y}}, \{(\varphi_i)_{\fml{Y}}\}_{i\in I})$ is   
    the formal Zariski colimit of the diagram
    $\{(\fml{Z}_i)_{\fml{Y}}\}_{i\in I}$; 
  \item[(4)] \textbf{formal geometric colimit} if
    it is a universal formal Zariski and a formal weak geometric colimit of the
    diagram
    $\{\fml{Z}_i\}_{i\in 
      I}$;
  \item[(5)] \textbf{uniform formal geometric colimit} if for any
    adic flat formal $\fml{Z}$-scheme $\fml{Y}$, the data
    $({\fml{Y}},\{(\varphi_i)_{\fml{Y}}\}_{i\in I})$ is a
    formal geometric colimit of the diagram
    $\{(\fml{Z}_i)_{\fml{Y}}\}_{i\in I}$.    
  \end{enumerate}
\end{defn}
We have two lemmas which are the analogues
of Lemmas \ref{lem:fin_loc_col} and \ref{lem:sch_wgeom_colim} for
formal schemes.
\begin{lem}\label{lem:wgeom_fml_col}
  Let $\{\fml{Z}_i\}_{i\in I}$ be a diagram of formal schemes.
  Then every weak geometric colimit $(\fml{Z},\{\varphi_i\}_{i\in
    I})$ is a colimit in the category of topologically \emph{locally ringed}
  spaces and, consequently, a colimit in the category $\FMLSCHABS$.
\end{lem}
\begin{lem}\label{lem:fin_loc_fml_col}
  Consider a \emph{finite} diagram of locally noetherian
  formal schemes $\{\fml{Z}_i\}_{i\in I}$, and a locally noetherian
  formal scheme $\fml{Z}$, together with \emph{finite} morphisms
  $\varphi_i \colon \fml{Z}_i \to \fml{Z}$. If the data
  $(\fml{Z},\{\varphi_i\}_{i\in I})$ is a formal geometric colimit of
  the diagram $\{\fml{Z}_i\}_{i\in I}$, then it is a uniform 
  formal geometric colimit.
\end{lem}
\begin{proof}
  Let $\varpi\colon \fml{Y} \to \fml{Z}$ be an adic flat
  morphism of locally noetherian formal schemes.
  It remains to show that the canonical map $\vartheta_{\fml{Y}} \colon
  \Orb_{\fml{Y}} \to \varprojlim_i
  [(\varphi_i)_{\fml{Y}}]_*\Orb_{(\fml{Z}_i)_{\fml{Y}}}$ is a
  topological isomorphism. Since $\varphi_i$ is finite for all $i$,
  and we are taking a finite limit, we conclude that it suffices to
  show that $\vartheta_{\fml{Y}}$ is an isomorphism of coherent
  $\Orb_{\fml{Y}}$-modules. By hypothesis, the map
  $\vartheta_{\fml{Z}} \colon \Orb_{\fml{Z}} \to \varprojlim_i
  (\varphi_i)_*\Orb_{\fml{Z}_i}$ is an isomorphism, and since $\varpi$ is adic
  flat, $\varpi^*$ is an exact functor from coherent
  $\Orb_{\fml{Z}}$-modules to coherent $\Orb_{\fml{Y}}$-modules, thus
  commutes with finite limits. Hence, we see that
  $\vartheta_{\fml{Y}}$ factors as the following sequence of isomorphisms:
  \[
  \Orb_{\fml{Y}} \cong \varpi^*\Orb_{\fml{Z}} \cong
  \varprojlim_i
  \varpi^*[(\varphi_i)_*\Orb_{\fml{Z}_i}] \cong
  \varprojlim_i [(\varphi_i)_{\fml{Y}}]_*\Orb_{(\fml{Z}_i)_{\fml{Y}}}.\qedhere
  \]
\end{proof}
\begin{defn}
  Fix a locally noetherian formal scheme $\fml{X}$ and let
  $\{\fml{Z}_i\}_{i\in I}$ be a diagram
  in $\FMLSCH{\fml{X}}$, where each formal scheme
  $\fml{Z}_i$ is locally noetherian. Let
  $\fml{Z}\in\FMLSCH{\fml{X}}$ also be locally noetherian, 
  and suppose that we have compatible
  $\fml{X}$-morphisms $\varphi_i \colon\fml{Z}_i \to \fml{Z}$ that are
  topologically of finite type. Then we say
  that the data 
  $(\fml{Z},\{\varphi_i\}_{i\in I})$ is a \textbf{uniform categorical
    colimit in $\FMLSCH{\fml{X}}$} if for any 
  adic flat formal $\fml{X}$-scheme
  $\fml{Y}$, the data
  $(\fml{Z}_{\fml{Y}},\{(\varphi_i)_{\fml{Y}}\}_{i\in I})$ is the
  categorical colimit of the diagram
  $\{(\fml{Z}_i)_{\fml{Y}}\}_{i\in I}$ in $\FMLSCH{\fml{Y}}$. 
\end{defn}
\begin{defn}
  Let $\fml{X}$ be a formal scheme. Define
  $\qfs{\fml{X}}$ to be the category whose objects are adic,
  quasi-finite, and separated maps $(\fml{Z} \xrightarrow{\sigma}
  \fml{X})$. A morphism in $\qfs{\fml{X}}$ is an $\fml{X}$-morphism
  $\fml{f} \colon (\fml{Z} \xrightarrow{\sigma} \fml{X}) \to (\fml{Z}'
  \xrightarrow{\sigma'} \fml{X})$.  
\end{defn}
For a scheme $X$, and a closed subset $|V| \subset |X|$, we
define the \textbf{completion functor} 
\[
c_{X,|V|} \colon \SCH{X} \to \FMLSCH{\cmpl{X}_{/|V|}},\quad (Z \xrightarrow{s}
X) \mapsto 
\bigl(\cmpl{Z}_{/s^{-1}|V|} \to \cmpl{X}_{/|V|}\bigr).
\]
Note that restricting $c_{X,|V|}$ to $\qfs{X}$ has essential image
contained in $\qfs{\cmpl{X}_{/|V|}}$.
\begin{thm}\label{thm:fml_pushouts_commute}
  Let $X$ be a locally noetherian scheme, let $|V|
  \subset 
  |X|$ be a closed subset, and let $\{Z_i\}_{i\in I}$
  be a diagram in $\qfsfin{X}$. If the scheme $Z$
  denotes the categorical colimit of the diagram in $\SCH{X}$, which
  exists by Theorem \ref{thm:pushouts}, then
  $c_{X,|V|}(Z)$ is a uniform categorical 
  and uniform formal geometric colimit of the diagram
  $\{c_{X,|V|}(Z_i)\}_{i\in I}$
  in $\FMLSCH{\cmpl{X}_{/|V|}}$. 
\end{thm}
\begin{proof}
Let $\fml{X}=\cmpl{X}_{/|V|}$,
$\fml{Z}=c_{X,|V|}(Z)$, and $\fml{Z_i}=c_{X,|V|}(Z_i)$.
By Lemmas \ref{lem:wgeom_fml_col} and \ref{lem:fin_loc_fml_col}, it
suffices to show that $\fml{Z}$ is a formal geometric
colimit of the diagram $\{\fml{Z_i}\}_{i\in I}$. We first show
that $\fml{Z}$ is a universal formal Zariski colimit. Let
$\varpi \colon \fml{Y} \to \fml{X}$ be an adic morphism of locally
noetherian formal schemes. Let
$\mathscr{I}$ be a coherent sheaf of radical ideals defining
$|V| \subset |X|$ so that $(\varpi^{-1}\mathscr{I})\Orb_{\fml{Y}}$ is an
ideal of definition of $\fml{Y}$. Let $X_0 = V(\mathscr{I}) \subset
X$ and $Y_0 = \fml{Y} \times_{\fml{X}} X_0$.
Then by Theorem \ref{thm:pushouts} the map of topological spaces
$\varinjlim_{i\in I}|Z_i\times_X Y_0| \to
|Z\times_X Y_0|$ is a homeomorphism. Noting that $|\fml{Z}_{\fml{Y}}|
= |Z\times_X Y_0|$ and $|\fml{Z_i}_{\fml{Y}}|
= |Z_i\times_X Y_0|$, we conclude that $\fml{Z}$ is a universal formal
Zariski colimit and it remains to show that $\fml{Z}$ has the
correct functions. 

Let $\phi_i \colon Z_i \to Z$ denote the canonical morphisms.
By Theorem \ref{thm:pushouts} we
have an isomorphism of sheaves of rings $\epsilon \colon \Orb_Z \to
\varprojlim_i (\phi_i)_*\Orb_{Z_i}$. Hence, since $\phi_i$
is finite, and completion is exact on coherent
modules \cite[I.10.8.9]{EGA}, by \cite[I.10.8.8(i)]{EGA} we have an
isomorphism of coherent $\Orb_{\fml{Z}}$-algebras:
\[
\Orb_{\fml{Z}} \xrightarrow{\epsilon^\wedge}
(\varprojlim_i \phi_{i*}\Orb_{Z_i})^\wedge
\cong \varprojlim_i (\phi_{i*}\Orb_{Z_i})^\wedge
\cong \varprojlim_i \cmpl{\phi_{i}}_*\Orb_{\fml{Z_i}}.
\]
It remains to show that this isomorphism is topological (where we
endow the right side with the limit topology). A general fact
here is that the topology on the right is the subspace topology of the
product topology on $\prod_{i\in I} \cmpl{\phi_i}_*\Orb_{\fml{Z_i}}$.
Since $\phi_i$ is finite, Krull's Theorem
\cite[$0_{\mathrm{I}}$.7.3.2]{EGA} implies that the limit and adic
topologies coincide. The result follows. 
\end{proof}
\subsection{Completions of algebraic stacks}
Here we observe that the colimits constructed in Theorem
\ref{thm:pushouts} remain colimits after completing along a closed
subset. To avoid developing the theory of formal algebraic
stacks, we will work with adic systems of algebraic stacks. This has
the advantage of being elementary as well as sufficient for our
purposes. 

Let $X$ be a locally noetherian algebraic stack and suppose that $|V|
\subseteq |X|$ is a closed subset, defined by a coherent
$\Orb_X$-ideal $I$. For each $n\geq 0$ let $X_n = V(I^{n+1})$ and define:
\[
\qfs{\cmpl{X}_{/|V|}} := \varprojlim_n \qfs{X_n}. 
\]
This is an abuse of notation, and for that we apologize. We firmly
believe, however, that this notation will be sufficiently convenient
to outweigh any potential confusion that may arise. Also, there is a
completion functor: 
\[
c_{X,|V|} \colon \qfs{X} \to \qfs{\cmpl{X}_{/|V|}},\quad (Z \to X) \mapsto
(Z\times_X X_n \to X_n)_{n\geq 0}.
\]
We conclude this section with a corollary,
which is an immediate consequence of
Theorem \ref{thm:pushouts}, smooth descent, and Theorem \ref{thm:fml_pushouts_commute}. 
\begin{cor}\label{cor:fml_pushouts_commute_stks}
  Let $X$ be a locally noetherian algebraic stack. Suppose that $|V|
  \subseteq 
  |X|$ is a closed subset. Let $\{Z_i\}_{i\in I}$ be a diagram
  in $\qfsfin{X}$ and let $Z\in\qfs{X}$ 
  be the categorical colimit, which
  exists by Theorem \ref{thm:pushouts}. Then
  $c_{X,|V|}(Z)$ is a categorical colimit of the diagram
  $\{c_{X,|V|}(Z_i)\}_{i\in I}$
  in $\qfs{\cmpl{X}_{/|V|}}$, and remains so after flat and locally
  noetherian base change on $X$. 
\end{cor}


\section{Generalized Stein
  factorizations}\label{ch:stein_factorizations}
A morphism of schemes $f \colon X \to Y$ is \textbf{Stein} if the morphism
$f^\sharp \colon \Orb_Y \to f_*\Orb_X$ is an isomorphism. If $f$ is a
proper morphism of locally noetherian schemes that is Stein, then
Zariski's Connectedness Theorem \cite[III.4.3.2]{EGA} implies that $f$ 
is surjective with geometrically connected fibers.  Note that if $f$ is
quasi-compact and quasi-separated then it factors as:
\[
X \xrightarrow{r} \widetilde{X} \xrightarrow{\tilde{f}} Y,
\]
where $r$ is Stein and $\tilde{f}$ is affine. Indeed, we simply take
$\widetilde{X}=\spec_Y(f_*\Orb_X)$. In general, this factorization says little
about $f$. In the case where $Y$ is locally noetherian and $f$ is
proper, however, one obtains the well-known \textbf{Stein
  factorization} of $f$ \cite[III.4.3.1]{EGA}. In this case, $\tilde{f}$ is
finite and $r$ is proper and surjective with geometrically connected
fibers. Note that the Stein factorization is also unique and
compatible with flat base change on $Y$. We would like to generalize
these properties of the Stein factorization to non-proper morphisms. 

We also wish to point out that the above discussion is perfectly valid
for locally noetherian formal schemes. That is, if $\varphi \colon \fml{X} \to \fml{Y}$ is a
proper morphism of locally noetherian formal schemes, then it admits a
Stein factorization: $\fml{X} \xrightarrow{\rho} \widetilde{\fml{X}}
\xrightarrow{\tilde{\varphi}} \fml{Y}$, where $\tilde{\varphi}$ is finite and $\rho$
is proper and Stein. The Stein factorization is
compatible with (not necessarily adic) flat base change on
$\fml{Y}$ (Proposition~\ref{prop:fml_bc}) and $\rho$ has geometrically connected
fibers (Corollary~\ref{cor:zar_con_fml}).

In this section we will address the following question: suppose that $\varphi \colon \fml{X} \to \fml{Y}$ is an adic morphism of
locally noetherian formal schemes
that is locally of finite type. Is there a
factorization of $\varphi$ as $\fml{X} \xrightarrow{\rho} \widetilde{\fml{X}}
\xrightarrow{\tilde{\varphi}} \fml{Y}$, where $\tilde{\varphi}$ is locally
quasi-finite and $\rho$ is proper and Stein? If such a factorization
exists, we call it a \textbf{generalized Stein factorization}. A
desirable property of a generalized Stein factorization will be that
it is compatible with flat base change on $\fml{Y}$.  

Note that a necessary condition for $\varphi \colon \fml{X} \to \fml{Y}$ to
admit a generalized Stein factorization is that every $w\in |\fml{X}|$
lies in a proper connected component of the fiber
$\fml{X}_{\varphi(w)}$. We will call such morphisms \textbf{locally
  quasi-proper}. If $\varphi$ is separated (resp.\ quasi-compact), then
so is $\tilde{\varphi}$ in any generalized Stein factorization since
$\rho$ is proper and surjective. In \S\ref{sec:gen_stein} we prove
\begin{thm}\label{thm:pf_sch_adj}
  Let $\fml{Y}$ be a locally noetherian formal scheme. Suppose that
  $\varphi \colon \fml{X} \to \fml{Y}$ is locally quasi-proper and
  separated.  Then $\varphi$ admits a generalized Stein
  factorization $\fml{X} \xrightarrow{\rho} \widetilde{\fml{X}}
  \xrightarrow{\tilde{\varphi}} \fml{Y}$
  which is unique and compatible with flat base change on
  $\fml{Y}$. Moreover, $\tilde{\varphi}$ is separated and if $\varphi$ is
  quasi-compact, so is $\tilde{\varphi}$.
\end{thm}
In future work, we will prove Theorem
\ref{thm:pf_sch_adj} for algebraic 
stacks in order to compare our generalized Stein
factorizations with the connected component fibrations of
\cite[\S6.8]{MR1771927} and \cite[Thm.~2.5.2]{MR2820394}. We wish to
point out that Theorem \ref{thm:pf_sch_adj} should extend to a
reasonable theory of formal algebraic spaces or stacks. Unfortunately,
the only treatise on formal algebraic spaces that we are aware of is D.~Knutson's \cite[V]{MR0302647},
and this requires everything to be separated
\cite[V.2.1]{MR0302647}---an assumption we definitely do not
wish to impose. We are not aware of any written account of a theory of
formal algebraic stacks.
\begin{rem}
  We will construct the generalized Stein factorization by working
  \'etale-locally on $\fml{Y}$. If $f \colon X \to Y$ is a separated and locally
  quasi-proper morphism of schemes
  that is quasi-compact, then there is a more explicit
  construction of the generalized Stein factorization. If $\bar{Y}$ denotes
  the integral closure of $Y$ in $f_*\Orb_X$, then the image of $X\to \bar{Y}$
  is open and equals $\widetilde{X}$. This follows from
  Theorem~\ref{thm:pf_sch_adj} and~\cite[IV.8.12.3]{EGA}. Proving directly
  that the image is open, that $X\to\widetilde{X}$ is proper and that
  $\widetilde{X}\to Y$ is quasi-finite is possible but not trivial.
  This construction also has the deficiency that we were
  unable to easily adapt it to locally noetherian formal schemes. 
\end{rem}
\subsection{Uniqueness and base change of generalized Stein factorizations}
In this subsection we address the uniqueness and base change
assertions in Theorem \ref{thm:pf_sch_adj}. Both assertions will also
be important for the proof of the existence of generalized Stein
factorizations in Theorem \ref{thm:pf_sch_adj}. 
\begin{lem}\label{lem:uni_gs}
  Let $\varphi \colon \fml{X} \to \fml{Y}$ be a morphism of locally
  noetherian formal schemes that admits a 
  factorization $\fml{X} \xrightarrow{\rho} \widetilde{\fml{X}}
  \xrightarrow{\tilde{\varphi}} \fml{Y}$ where $\rho$ is proper and
  $\tilde{\varphi}$ is locally quasi-finite and separated. Then
  $\rho$ is Stein if and only if the natural map:
  \[
  \Hom_{\fml{Y}}(\widetilde{\fml{X}},\fml{Z}) \to \Hom_{\fml{Y}}(\fml{X},\fml{Z})
  \]
  is bijective for every locally quasi-finite and separated morphism
  $s \colon \fml{Z} \to \fml{Y}$. In particular, for
  locally quasi-proper and separated morphisms, the
  generalized Stein factorization, if it exists, is unique.
\end{lem}
\begin{proof}
  A $\fml{Y}$-morphism $\widetilde{\fml{X}} \to \fml{Z}$ is equivalent to a
  section of the projection $\fml{Z}\times_{\fml{Y}} \widetilde{\fml{X}} \to
  \widetilde{\fml{X}}$. In particular, we may now replace $\fml{Y}$ with
  $\widetilde{\fml{X}}$ and $\fml{Z}$ with $\fml{Z}\times_{\fml{Y}}
  \widetilde{\fml{X}}$. Thus we have reduced the result to the situation where
  $\varphi$ is proper and $\widetilde{\fml{X}} = \fml{Y}$. 

  Now let $\fml{X} \xrightarrow{\beta'} \st{\fml{Y}}(\fml{X})
  \xrightarrow{b} \fml{Y}$ be the Stein 
  factorization of $\varphi$ (which exists because $\varphi$ is
  proper). Fix a $\fml{Y}$-morphism $\alpha \colon \fml{X} \to
  \fml{Z}$. Since $\fml{Z}$ is separated over $\fml{Y}$ and $\varphi$
  is proper, it follows that $\alpha$ is proper. Thus, there is a Stein
  factorization of $\alpha\colon \fml{X} \xrightarrow{\alpha'}
  \st{\fml{Z}}(\fml{X}) \xrightarrow{a}  \fml{Z}$. Note, however, that
  $\gamma \colon \st{\fml{Z}}(\fml{X}) \to \fml{Y}$ is
  quasi-finite and proper, thus finite \cite[III.4.8.11]{EGA}. Next
  observe that we have natural isomorphisms of coherent sheaves of
  $\Orb_{\fml{Y}}$-algebras:
  \[
  b_*\Orb_{\st{\fml{Y}}(\fml{X})} \cong b_*\beta'_*\Orb_{\fml{X}}
  \cong \gamma_*\alpha'_*\Orb_{\fml{X}} \cong
  \gamma_*\Orb_{\st{\fml{Z}}(\fml{X})}.
  \]
  Hence, there exists a unique $\fml{Y}$-isomorphism $\delta \colon
  \st{\fml{Y}}(\fml{X}) \to \st{\fml{Z}}(\fml{X})$ that is compatible
  with the data. If $\varphi$ is Stein, then $\st{\fml{Y}}(\fml{X}) =
  \fml{X}$ and $b=\Identity{\fml{X}}$. We deduce that there is a uniquely induced
  $\fml{Y}$-morphism $\fml{Y} \xrightarrow{\delta}
  \st{\fml{Z}}(\fml{X}) \xrightarrow{a} \fml{Z}$. Conversely, we have
  a Stein factorization $\fml{X} \to \st{\widetilde{\fml{X}}}(\fml{X}) \to
  \widetilde{\fml{X}}$. By the Stein case already considered, we see that
  $\st{\widetilde{\fml{X}}}(\fml{X})$ and $\widetilde{\fml{X}}$ both satisfy
  the same universal property for locally quasi-finite and separated
  $\fml{Y}$-schemes. The result follows.
\end{proof}
Fix a locally noetherian formal scheme $\fml{Y}$ and let
$\lqps{\fml{Y}}$ (resp.~$\lqfs{\fml{Y}}$) denote the category of
locally quasi-proper (resp.~locally quasi-finite) and separated
morphisms $\fml{X} \to \fml{Y}$. Lemma
\ref{lem:uni_gs} allows us to reinterpret Theorem \ref{thm:pf_sch_adj}
as the existence of a left adjoint: 
\[
\st{\fml{Y}} \colon \lqps{\fml{Y}} \to
\lqfs{\fml{Y}}
\]
to the inclusion $\lqfs{\fml{Y}} \subseteq \lqps{\fml{Y}}$ such that if $\fml{Z} \in \lqps{\fml{Y}}$, then:
\begin{enumerate}
\item the natural morphism $\fml{Z} \to \st{\fml{Y}}(\fml{Z})$ is proper;
\item if $\fml{Y}' \to \fml{Y}$ is flat, then we have a natural isomorphism:
  \[
  \st{\fml{Y}'}(\fml{Z}\times_{\fml{Y}}\fml{Y}') \cong \st{\fml{Y}}(\fml{Z})\times_{\fml{Y}} \fml{Y}'.
  \]
\end{enumerate}
We address property (2) immediately.
\begin{lem}\label{lem:st_bc_comp}
  Fix a cartesian diagram of locally noetherian formal schemes:
  \[
  \xymatrix@-0.8pc{\fml{X}' \ar[r]^{p'} \ar[d]_{\varphi'} & \fml{X}
    \ar[d]^{\varphi} \\ \fml{Y}' \ar[r]^p & \fml{Y}. }
  \]
  Suppose that $\varphi$ is separated and admits a generalized Stein
  factorization $\fml{X} \to \st{\fml{Y}}(\fml{X}) \to \fml{Y}$. Then
  $\varphi'$ admits a generalized Stein factorization $\fml{X}' \to
  \st{\fml{Y}'}(\fml{X}') \to \fml{Y}'$ and there is a naturally induced
  $\fml{Y}'$-morphism:
  \[
  h\colon\st{\fml{Y}'}(\fml{X}') \to \st{\fml{Y}}(\fml{X})\times_{\fml{Y}} \fml{Y}',
  \]
  which is a finite universal homeomorphism. In addition, if $p$ is
  flat, then the morphism $h$ is an isomorphism.
\end{lem}
\begin{proof}
  The morphism $\varphi'$ factors as $\fml{X}' \xrightarrow{\gamma}
  \st{\fml{Y}}(\fml{X}) \times_{\fml{Y}} \fml{Y}' \xrightarrow{\sigma}
  \fml{Y}'$ where $\sigma$ is locally quasi-finite and separated and
  $\gamma$ is proper. Next observe that the Stein factorization of
  $\gamma$ is the generalized Stein factorization of $\varphi'$. Thus
  $\gamma$ factors as:
  \[
  \fml{X}' \to \st{\fml{Y}'}(\fml{X}') \xrightarrow{h}
  \st{\fml{Y}}(\fml{X})\times_{\fml{Y}}\fml{Y}'.
  \]
  By Zariski's Connectedness Theorem (Corollary \ref{cor:zar_con_fml}) $\fml{X}'
  \to \st{\fml{Y}'}(\fml{X}')$ and $\gamma$ have geometrically
  connected fibers. It follows that $h$ is a finite universal
  homeomorphism (it is finite with geometrically connected fibers). In
  addition, if $p$ is flat, then since cohomology commutes with flat
  base change (Proposition \ref{prop:fml_bc}), $\gamma$ is already
  Stein, and so $h$ is an isomorphism.  
\end{proof}
\subsection{Local decompositions}
Before we prove Theorem \ref{thm:pf_sch_adj} it will be necessary to
understand the \'etale local structure of locally quasi-proper
morphisms. We accomplish this by generalizing the well-known structure
results that are available for locally quasi-finite and separated morphisms
\cite[IV.18.5.11c, IV.18.12.1]{EGA}. 
\begin{prop}\label{prop:decomp}
  Let $\varphi \colon \fml{X} \to \fml{Y}$ be an adic morphism of locally
  noetherian formal schemes that is locally of finite
  type and separated. Let $y\colon\spec(k)\to \fml{Y}$ be a point and suppose that  
  $Z$ is a connected component of the fiber $\fml{X}_y$. If $Z$ is
  proper, then there exists an adic \'etale neighborhood $(\fml{Y}',y') \to
  (\fml{Y},y)$ and an open and closed immersion $i
  \colon \fml{X}_Z
  \to \fml{X}\times_{\fml{Y}} \fml{Y}'$ where $\fml{X}_Z \to \fml{Y}'$ is proper and $(\fml{X}_Z)_{y'}\cong Z$.
\end{prop}
\begin{proof}
  We first prove the result in the setting of a morphism of locally
  noetherian \emph{schemes} $f \colon X \to Y$. We now immediately reduce
  to the case where $Y$ is affine and $f$ is 
  quasi-compact. Let $k_0\subseteq k$ denote the separable closure of the
  residue field $k(y)$ in $k$. Then the resulting morphism $X_y\to X_{k_0}$ induces a bijection
  on connected components~\cite[IV.4.3.2]{EGA}. So, we may replace $k$
  with $k_0$ and $Z$ by the corresponding connected component in $X_{k_0}$. 
  By standard limit and smearing out arguments
  \cite[IV.8]{EGA} we further reduce to the case where $Y$ is local
  and henselian with closed point $y$ and residue field $k$.

  By Chow's Lemma \cite[II.5.6.1]{EGA}, there exists a quasi-projective
  $Y$-scheme $W$ and a proper and surjective $Y$-morphism $p \colon W \to
  X$. If the result has been proved for $W \to Y$, then we claim that
  this implies the result for $f \colon X \to Y$. Indeed, each connected
  component of $p^{-1}(Z)$ is proper, thus we may write
  $W=W_{p^{-1}(Z)} \amalg W'$ where $W_{p^{-1}(Z)}\to Y$ is proper,
  and the connected components of $W_{p^{-1}(Z)}$
  coincide with the connected components of $(W_{p^{-1}(Z)})_y = p^{-1}(Z)$.
  Now take $X_Z =
  p(W_{p^{-1}(Z)})$ and $X' = p(W')$. Then $X_Z$ is proper over $Y$ and it remains to show that $X_Z
  \cap X' = \emptyset$. Note that $X_Z \cap X'$ is a closed subset of
  $X_Z$. If it is non-empty, then it must have non-empty intersection
  with $Z$. But this would imply that $Z \cap X' \neq \emptyset$,
  hence $p^{-1}(Z) \cap W'\neq \emptyset$---a contradiction. We
  deduce that it remains to prove the result when $X$ is
  also quasi-projective over $Y$.

  In this case, there is an open and dense immersion $i \colon X
  \hookrightarrow \bar{X}$ where $\bar{X}$ is a projective
  $Y$-scheme. Note that if $V$ is a connected component of $\bar{X}_y$
  and $V\cap Z \neq \emptyset$, then $V\cap Z = V$ since $Z$ is proper. By
  Corollary \ref{cor:hens_pair_prop}, there is also a bijection between the set of
  connected components of $\bar{X}_y$ and $\bar{X}$. Set
  $X_Z$ to be the union of those connected components of
  $\bar{X}$ that meet $Z$. Observe that $X_Z \cap X$ is an open
  subset of $X_Z$ which contains $(X_Z)_y = Z$. Since $X_Z \to Y$ is
  proper and $Y$ is a local scheme, we have that $X_Z \cap X=X_Z$ and we
  conclude that $X_Z \subseteq X$. In particular, we readily deduce
  that $X_Z$ is an open and closed subset of $X$ and we have the
  claimed decomposition $X = X_Z \amalg X'$ with $(X_Z)_y = Z$. 

  Now let $\varphi \colon \fml{X} \to \fml{Y}$ be as in the
  Proposition. Let $f \colon X \to Y$ denote the induced morphism on the
  underlying schemes. By what we have proven for schemes, there is
  an \'etale neighborhood $(Y',y') \to (Y,y)$ and an open and closed immersion $i \colon X_Z \to X\times_Y Y'$ where
  $X_Z \to Y'$ is proper and $(X_Z)_{y'} \cong Z$. By
  \cite[IV.18.1.2]{EGA}, there is a unique lift of the \'etale
  neighborhood $(Y',y') \to (Y,y)$ to an adic \'etale neighborhood
  $(\fml{Y}',y') \to (\fml{Y},y)$ whose underlying morphism of pointed
  schemes coincides with $(Y',y') \to (Y,y)$. The result follows. 
\end{proof}
\begin{rem}\label{rem:non-noeth}
  In future work we will prove Proposition \ref{prop:decomp} for
  non-separated morphisms of non-noetherian algebraic
  stacks. We do not treat this here
  as the necessary reformulations in the
  non-locally-separated case would take us too far afield. 
\end{rem}
We conclude this subsection with a trivial, but important, corollary.
\begin{cor}\label{cor:decomp}
  Let $\varphi \colon \fml{X} \to \fml{Y}$ be a morphism of locally
  noetherian formal schemes that is separated and locally
  quasi-proper. Then there exists an adic \'etale morphism
  $\fml{Y}' \to \fml{Y}$, and an open and closed immersion $i \colon \fml{U}
  \hookrightarrow \fml{X}\times_{\fml{Y}} \fml{Y}'$ such that the
  composition $\fml{U} \to \fml{X}\times_{\fml{Y}} \fml{Y}' \to
  \fml{X}$ is surjective and the composition $\fml{U} \to
  \fml{X}\times_{\fml{Y}} \fml{Y}' \to \fml{Y}'$ is proper.
\end{cor}
Thus Corollary \ref{cor:decomp} implies that any
separated and locally quasi-proper morphism $\varphi \colon \fml{X} \to
\fml{Y}$ admits an \'etale slice 
that has a generalized Stein factorization. Thus in order to
construct the generalized Stein factorization of $\varphi$, we can use
\'etale descent. The problem now is thus the description of 
the relevant descent data. In order to this, we will address
a more general problem in \S\ref{subsec:et_mor}: for a proper and Stein
morphism, which \'etale morphisms are pulled back from the base? 
\subsection{Pullbacks of \'etale morphisms}\label{subsec:et_mor}
Let $\varphi \colon \fml{X} \to \fml{Y}$ be a proper and Stein morphism of
locally noetherian formal schemes. In this subsection we will identify
those \'etale morphisms to $\fml{X}$ that are pulled back from
$\fml{Y}$. 

So, let $\fml{S}$ be a
locally noetherian formal scheme. Define $\Etsep{\fml{S}}$ to be the
category of morphisms $(\fml{V} \to \fml{S})$ that are adic \'etale
and separated. Also, define $\Of{\fml{S}}$ to be the set of open and
closed immersions into $\fml{S}$. If $S\subseteq \fml{S}$ denotes the underlying scheme
of $\fml{S}$, then the natural map $\Of{\fml{S}} \to \Of{S}$ is
bijective. More generally, an easy consequence of \cite[IV.18.1.2]{EGA} is
that the functor $\Etsep{\fml{S}} \to \Etsep{S}$ is an equivalence of
categories. 

Let $\varphi \colon \fml{X} \to \fml{Y}$ be an adic morphism of
locally noetherian formal schemes. There is an induced functor 
\[
\varphi^* \colon \Etsep{\fml{Y}} \to \Etsep{\fml{X}},\quad (\fml{V} \to
\fml{Y}) \mapsto (\fml{V}\times_{\fml{Y}} \fml{X} \to \fml{X}).
\]
Let $\Etadsep{\fml{X}}{\fml{Y}}$ be the full subcategory of
$\Etsep{\fml{X}}$ with objects those $\fml{W} 
\to \fml{X}$ such that for any geometric point $\bar{y}$ of $\fml{Y}$,
and any connected component $Z$ of $\fml{W}_{\bar{y}}$, the induced
morphism $Z \to \fml{X}_{\bar{s}}$ is an 
open and closed immersion. Note that $\Of{\fml{X}} \subseteq
\Etadsep{\fml{X}}{\fml{Y}}$. In the following Lemma, we characterize the
essential image of $\varphi^*$. 
\begin{lem}\label{lem:et_gen_st_fml}
  Let $\varphi \colon \fml{X} \to \fml{Y}$ be a proper and Stein morphism
  of locally noetherian formal schemes. 
  \begin{enumerate}
  \item\label{item:et_gen_st_fml:char} The functor $\varphi^* \colon
    \Etsep{\fml{Y}} \to \Etsep{\fml{X}}$ is fully faithful with image
    $\Etadsep{\fml{X}}{\fml{Y}}$.
  \item\label{item:et_gen_st_fml:of} The induced map $\Of{\fml{Y}} \to
    \Of{\fml{X}}$ is bijective.
  \item\label{item:et_gen_st_fml:stein} If $\fml{W} \in
    \Etadsep{\fml{X}}{\fml{Y}}$, then it admits a generalized Stein
    factorization $\fml{W} \to \st{\fml{Y}}(\fml{W}) \to \fml{Y}$,
    with $\st{\fml{Y}}(\fml{W}) \to \fml{Y}$ \'etale and separated.
    Moreover, $\fml{W}=\fml{Y}\times_{\fml{X}}\st{\fml{Y}}(\fml{W})$.
  \end{enumerate}
\end{lem}
\begin{proof}
  For \itemref{item:et_gen_st_fml:char} and
  \itemref{item:et_gen_st_fml:of}, by passing to the underlying
  morphism of schemes we obtain, via Zariski's Connected Theorem
  (Corollary \ref{cor:zar_con_fml}), a proper and surjective morphism
  of locally noetherian schemes $f \colon X \to Y$ with geometrically
  connected fibers. By Proposition \ref{prop:et_prop_geom_conn} and
  Remark \ref{rem:et_sep_clopen}, the functor $f^* \colon \Etsep{Y} \to
  \Etsep{X}$ is fully faithful with image $\Etadsep{X}{Y}$ and the
  induced map $\Of{\fml{Y}} \to \Of{\fml{X}}$ is bijective. The claim
  \itemref{item:et_gen_st_fml:stein} follows from
  \itemref{item:et_gen_st_fml:char} and Lemma \ref{lem:st_bc_comp}. 
\end{proof}
\subsection{Existence of generalized Stein factorizations}\label{sec:gen_stein}
In this subsection we will prove 
Theorem \ref{thm:pf_sch_adj}. Before we can accomplish this we require
one more Lemma. This Lemma is well-known for schemes, though requires
some care for formal schemes. 
\begin{lem}\label{lem:fml_eq_relns_qf}
  Let $\fml{X}$ be a locally noetherian formal scheme and let $\fml{U}
  \to \fml{X}$ be an adic, locally quasi-finite, and separated
  morphism. Suppose that $[\fml{R} \rightrightarrows \fml{U}]$ is an
  adic flat equivalence relation over $\fml{X}$ such that $\fml{R} \to
  \fml{U}\times_{\fml{X}} \fml{U}$ is a closed immersion. Then this
  equivalence relation has a formal geometric quotient
  $\fml{Z}$ that is adic, locally quasi-finite, and separated over
  $\fml{X}$. Moreover, taking the quotient commutes with arbitrary (not
  necessarily adic) base change, the quotient map $\fml{U}\to \fml{Z}$ is adic
  and faithfully flat, and
  $\fml{R}=\fml{U}\times_{\fml{Z}} \fml{U}$.
\end{lem}
\begin{proof}
  Let $\mathscr{I}$ be an ideal of definition for $\fml{X}$ and denote
  $X_{\mathscr{I}} = V(\mathscr{I})$, $U_{\mathscr{I}} = \fml{U}
  \times_{\fml{X}} X_{\mathscr{I}}$, and 
  $R_{\mathscr{I}} = \fml{R}\times_{\fml{X}} X_{\mathscr{I}}$. Then,
  the hypotheses ensure that 
  $[R_{\mathscr{I}} \rightrightarrows
    U_{\mathscr{I}}]$ is an fppf
  equivalence relation over $X_{\mathscr{I}}$. We observe that
  the quotient of this fppf equivalence relation in the category of
  algebraic spaces is a scheme $Z_{\mathscr{I}}$, as it is
  locally quasi-finite and separated over 
  $X_{\mathscr{I}}$ \cite[II.6.16]{MR0302647}. 
  Note that for any other ideal of definition
  $\mathscr{J} \supset \mathscr{I}$ there is a natural isomorphism
  $Z_{\mathscr{I}} \times_{\fml{X}} X_{\mathscr{J}} \cong 
  Z_{\mathscr{J}}$. Hence, the directed system
  $\{Z_{\mathscr{I}}\}_{\mathscr{I}}$ is adic over $\fml{X}$. Taking
  the colimit of this system in the category of topologically ringed
  spaces produces a locally noetherian formal scheme $\fml{Z}$, which
  is adic, locally quasi-finite, and separated over $\fml{X}$. It is
  now immediate that
  $\fml{Z}$ is 
  a formal geometric quotient of the given equivalence
  relation. That $\fml{Z}$ has the stated properties is an easy
  consequence of the corresponding results for its truncations
  $Z_{\mathscr{I}}$. 
\end{proof}
We now proceed to the proof of Theorem \ref{thm:pf_sch_adj}.
\begin{proof}[Proof of Theorem \ref{thm:pf_sch_adj}]
  The latter claims are consequences of Lemmas \ref{lem:uni_gs}
  and \ref{lem:st_bc_comp}, thus it remains to address the existence
  of generalized Stein factorizations.

  Note that if $\fml{Y} \to \fml{T}$ is locally quasi-finite and
  separated, then $\st{\fml{Y}}(\fml{X}) \cong \st{\fml{T}}(\fml{X})$
  over $\fml{T}$. Moreover, $\Etadsep{\fml{X}}{\fml{Y}} \homotopic
  \Etadsep{\fml{X}}{\fml{T}}$. Thus if $\varphi \colon \fml{X} \to \fml{Y}$ factors as
  $\fml{X} \to \fml{Y}' \to \fml{Y}$, where $\fml{X} \to \fml{Y}'$ is
  proper and $\fml{Y}' \to \fml{Y}$ is locally quasi-finite and separated, then the Stein factorization of $\fml{X} \to \fml{Y}'$ gives
  the generalized Stein factorization of $\fml{X} \to \fml{Y}$. 
  
  So, we fix a locally quasi-proper and separated morphism $\varphi \colon
  \fml{X} \to \fml{Y}$. By Corollary \ref{cor:decomp} there is an adic,
  \'etale, and separated morphism $\fml{Y}' \to 
  \fml{Y}$, together with an open and closed immersion $i \colon \fml{U}
  \hookrightarrow \fml{X}':=\fml{X}\times_{\fml{Y}} \fml{Y}'$, such
  that the induced morphism $\fml{U} \to \fml{Y}'$ is proper, and the
  induced morphism $\fml{U} \to \fml{X}$ is adic, \'etale, separated,
  and surjective. By the remarks above we obtain a generalized Stein
  factorization $\fml{U} \to \st{\fml{Y}}(\fml{U}) \to \fml{Y}$ (note
  that $\st{\fml{Y}}(\fml{U}) = \st{\fml{Y}'}(\fml{U})$). 

  Let $\fml{Y}'' = \fml{Y}'\times_{\fml{Y}} \fml{Y}'$ and take $s_1$,
  $s_2 \colon \fml{Y}'' \to \fml{Y}'$ to denote the two projections.  Let
  $\fml{X}'' = \fml{X}'\times_{\fml{X}} \fml{X}'$ and denote by $t_1$
  and $t_2$ the two projections. For $j=1$ and $2$ let $\fml{R}_j$
  denote the pullback of $\fml{U}$ along $t_j$. Note that the
  morphisms $i_j \colon \fml{R}_j \to \fml{X}''$ are open and closed
  immersions. Let
  $\fml{R} = \fml{R}_1 \cap \fml{R}_2=\fml{U}\times_\fml{X} \fml{U}$.
  Then we obtain
  an adic \'etale equivalence relation $[\fml{R} \rightrightarrows
  \fml{U}]$ with quotient $\fml{X}$.

  Note that both morphisms $\fml{R} \to \fml{U}$ belong to
  $\Etadsep{\fml{U}}{\fml{Y}}$. The morphism $\fml{U}
  \to \st{\fml{Y}}(\fml{U})$ is also proper and Stein. By Lemma
  \ref{lem:et_gen_st_fml}\itemref{item:et_gen_st_fml:char} we see that the functor
  $\Etsep{\st{\fml{Y}}(\fml{U})} \to \Etadsep{\fml{U}}{\fml{Y}}$ is an
  equivalence of categories (note that $\Etadsep{\fml{U}}{\fml{Y}}
  \homotopic
  \Etadsep{\fml{U}}{\st{\fml{Y}}(\fml{U})}$). Functoriality, together
  with Lemma \ref{lem:et_gen_st_fml}\itemref{item:et_gen_st_fml:stein}, produces an adic \'etale \emph{groupoid} $[\st{\fml{Y}}(\fml{R}) \rightrightarrows
  \st{\fml{Y}}(\fml{U})]$ in $\lqfs{\fml{Y}}$ which pulls back to the
  equivalence relation $[\fml{R} \rightrightarrows \fml{U}]$. We will now
  verify that
  $\st{\fml{Y}}(\fml{R}) \to
  \st{\fml{Y}}(\fml{U})\times_{\fml{Y}} \st{\fml{Y}}(\fml{U})$ is a closed
  immersion.

  Note that for $j=1$ and
  $2$ we have that $\fml{R} \in
  \Of{\fml{R}_j}$. By Lemma
  \ref{lem:et_gen_st_fml}(\ref{item:et_gen_st_fml:of},\ref{item:et_gen_st_fml:stein}) we thus see
  that $\st{\fml{Y}}(\fml{R}) \in \Of{\st{\fml{Y}}(\fml{R}_j)}$. We
  also know that $\st{\fml{Y}}(\fml{R}_j)$ is the pullback of
  $\st{\fml{Y}}(\fml{U}) \to \fml{Y}'$ along $s_j \colon \fml{Y}'' \to
  \fml{Y}'$ (Lemma \ref{lem:st_bc_comp}). Thus we obtain a $\fml{Y}$-isomorphism:
  \[
  \st{\fml{Y}}(\fml{R}_1) \times_{\fml{Y}''} \st{\fml{Y}}(\fml{R}_2)
  \to \st{\fml{Y}}(\fml{U}) \times_{\fml{Y}} \st{\fml{Y}}(\fml{U}). 
  \]
  We have already seen, however, that the map $\st{\fml{Y}}(\fml{R})
  \to \st{\fml{Y}}(\fml{R}_1)$ is an open and closed immersion.
  Since everything is separated, it follows that the natural map
  $\st{\fml{Y}}(\fml{R}) \to \st{\fml{Y}}(\fml{R}_1)\times_{\fml{Y}''}
  \st{\fml{Y}}(\fml{R}_2)=\st{\fml{Y}}(\fml{U}) \times_{\fml{Y}}
  \st{\fml{Y}}(\fml{U})$ is a closed immersion. Thus, we have an adic
  \'etale \emph{equivalence relation} $[\st{\fml{Y}}(\fml{R})
  \rightrightarrows \st{\fml{Y}}(\fml{U})]$ in $\lqfs{\fml{Y}}$
  satisfying the hypotheses of Lemma \ref{lem:fml_eq_relns_qf}.

  Let $\widetilde{\fml{X}}\in\lqfs{\fml{Y}}$ be the quotient of this equivalence
  relation. Since $\st{\fml{Y}}(\fml{U})\to \widetilde{\fml{X}}$ is adic \'etale
  and $\st{\fml{Y}}(\fml{R})=\st{\fml{Y}}(\fml{U})\times_{\widetilde{\fml{X}}}
  \st{\fml{Y}}(\fml{U})$, the proper and Stein morphism
  $\fml{U}\to\st{\fml{Y}}(\fml{U})$, with descent data given by
  $\fml{R}\to\st{\fml{Y}}(\fml{R})$, descends to a proper and Stein morphism
  $\fml{X}\to \widetilde{\fml{X}}$. This is the generalized Stein factorization of
  $\varphi$ by Lemma \ref{lem:uni_gs}.
\end{proof}


\section{Adjunctions for quasi-finite spaces}\label{ch:qfs}
If $p \colon X' \to X$ is a morphism of algebraic stacks, then there is a
pullback functor $p^* \colon \qfs{X} \to 
\qfs{X'}$ given by $(Z \to X) \mapsto (Z\times_X X' \to
X')$. More generally, given a closed subset $|V| \subseteq |X|$, set
$|V'| = p^{-1}|V|$. Then there 
is an induced pullback functor $\cmpl{p}^*_{/|V|} \colon \qfs{\cmpl{X}_{/|V|}} \to
\qfs{\cmpl{X}'_{/|V'|}}$. In this
short section we will use generalized Stein factorizations to
construct \emph{left} adjoints to these functors when $p$ is
proper and schematic (i.e., those proper morphisms that are
representable by schemes). This is certainly true in greater
generality, 
though we limit ourselves to this situation for simplicity. 
\begin{defn}
  Suppose that $p \colon X' \to X$ is a proper morphism of locally noetherian algebraic stacks. Let $|V| \subseteq |X|$
  be a closed subset and let $|V'| = p^{-1}|V|$. Let
  $\qfp{\cmpl{X}'_{/|V'|}}{X}$ denote 
  the full subcategory of $\qfs{\cmpl{X}'_{/|V'|}}$ consisting
  of those adic systems $(Z_n \to X'_n)_{n\geq 0}$ such that the
  composition $Z_0 \to X'_0\to X_0$ has proper fibers. 
\end{defn}
We now have the following Theorem. 
\begin{thm}\label{thm:prop_push_fml}
  Suppose that $p \colon X' \to X$ is a proper and schematic morphism of locally noetherian algebraic stacks. Let $|V| \subseteq |X|$ be a
  closed subset and let $|V'| = p^{-1}|V|$. Then the functor
  $\cmpl{p}_{/|V|}^* \colon 
  \qfs{\cmpl{X}_{/|V|}} \to \qfs{\cmpl{X}'_{/|V'|}}$ factors
  through $ \qfp{\cmpl{X}'_{/|V'|}}{X}$ and admits a left
  adjoint:
  \[
  (\cmpl{p}_{/|V|})_! \colon   \qfp{\cmpl{X}'_{/|V'|}}{X} \to
  \qfs{\cmpl{X}_{/|V|}},
  \]
  which is compatible with locally noetherian and flat base change on
  $X$. Moreover, for $(Z \to X') \in \qfp{X'}{X}$, the induced
  natural map: 
  \[
  (\cmpl{p}_{/|V|})_!\circ c_{X',|V'|}(Z\to X') \to c_{X,|V|}
  \circ p_!(Z \to X) 
  \]
  is an isomorphism. 
\end{thm}
\begin{proof}
  We wish to point out that if $|V| = |X|$, then $\cmpl{X}_{/|V|} =
  X$ and $\qfs{\cmpl{X}_{/|V|}} = \qfs{X}$. Thus we denote
  $\cmpl{p}^*_{/|V|}$ by $p^*$ and the functor $p_!$ (which appears at
  the end of the statement) is a left adjoint to $p^*$. 

  Now the first claim about the factorization is trivial. For the
  existence and properties of $(\cmpl{p}_{/|V|})_!$, by smooth
  descent, it is sufficient to prove 
  the result in the case where $p \colon X' \to X$ is a morphism of locally
  noetherian schemes. Set $\fml{X} = \cmpl{X}_{/|V|}$ and $\fml{X}' =
  \cmpl{X}'_{/|V'|}$ and let $\pi \colon \fml{X}' \to \fml{X}$ be the
  induced morphism. 

  Let $(Z_n \to X'_n)_{n\geq 0} \in \qfp{\fml{X}'}{\fml{X}}$ which we may also
  view as a quasi-finite and separated morphism of locally noetherian
  formal schemes $\fml{Z} \to \fml{X}'$ such that the composition
  $\fml{Z} \to \fml{X}' \to \fml{X}$ is locally quasi-proper and
  separated.  By Theorem \ref{thm:pf_sch_adj}, there is a
  generalized Stein factorization $\fml{Z} \xrightarrow{\rho}
  \st{\fml{X}}(\fml{Z}) \to \fml{X}$ and $\st{\fml{X}}(\fml{Z}) \to
  \fml{X}$ is quasi-finite and separated. We set $\pi_!\fml{Z} =
  \st{\fml{X}}(\fml{Z})$ and note that $\pi_!$ is a left adjoint to
  $\pi^*$ by Lemma \ref{lem:uni_gs}. The compatibility with
  flat base change and completions is implied by the flat base change
  property of $\st{\fml{X}}$, together with the fact that
  $\fml{X} \to X$ is a (non-adic) flat morphism \cite[I.10.8.9]{EGA}. 
\end{proof}


\section{The Existence Theorem}\label{ch:exthm}
\begin{defn}
  Fix a morphism of algebraic stacks $\pi \colon X \to Y$ and a closed
  subset $|Y_0| \subseteq |Y|$. Let $|X_0| = \pi^{-1}|Y_0|$. Define
  $\qfc{\cmpl{X}_{/|X_0|}}{Y}$ to be the full
  subcategory of $\qfp{\cmpl{X}_{/|X_0|}}{Y}$ consisting of 
  those $(Z_n\to X)_{n\geq 0}$ such that the composition $Z_0 \to X
  \to Y$ is proper.
\end{defn}
Note that $\qfc{\cmpl{X}_{/|X_0|}}{Y}$ is also a full subcategory of
$\qfsfin{\cmpl{X}_{/|X_0|}}$. That is, every morphism in
$\qfc{\cmpl{X}_{/|X_0|}}{Y}$ is finite.

Throughout this section we fix a noetherian ring $R$ and an ideal $I
\subseteq R$, such that $R$ is $I$-adic. Let $S = \spec R$ and $S_n = \spec
(R/I^{n+1})$. Consider a morphism of algebraic stacks $\pi \colon X \to S$
that is locally of finite type. For $n\geq 0$ let $X_n = X\times_S S_n$. It will be
convenient for us to abbreviate the symbol $\cmpl{X}_{/|X_0|}$ to $\cmpl{X}$. Consider the 
$I$-adic completion functor:     
\[
\Psi_{\pi,I} \colon \qfc{X}{S} \longrightarrow \qfc{\cmpl{X}}{S},\quad 
(Z \to X) \mapsto (Z\times_X X_n \to X_n)_{n\geq 0}.
\]
In this section we prove an Existence Theorem, which
gives sufficient conditions on the morphism $\pi \colon X \to S$ for the
functor $\Psi_{\pi,I}$ to be an equivalence of categories. 

In the case that $\pi \colon X \to S$ is \emph{separated}, the category $\qfc{X}{S}$ is equivalent to the category of finite algebras over $X$ with $S$-proper
support. An immediate
consequence of \cite[Thm.~A.1]{MR2239345} is
that the functor
$\Psi_{\pi,I}$ is an equivalence in this situation. Note that
[\emph{loc.\ cit.}]\ also gives a straightforward proof that $\Psi_{\pi,I}$
is fully faithful when $\pi$ is no longer assumed to be
separated. Indeed, we have
\begin{lem}\label{lem:fullyfaithful}
  Let $\pi \colon X \to S$ be a morphism of algebraic stacks that is
  locally of finite type with quasi-compact and separated
  diagonal. Then the functor $\Psi_{\pi,I}$ is fully faithful.
\end{lem}
\begin{proof}
  We begin with the following general observation: let $f \colon V \to W$
  be a \emph{representable} morphism of algebraic $S$-stacks. Let
  $\Sec{}{V/W}$ denote the set of sections to $f$. Denote by $f_n \colon
  V_n \to W_n$ the pullback of $f$ along $S_n \hookrightarrow S$. If
  $W$ is a proper algebraic stack over $S$, and $f$ is also separated
  and locally of finite type, then the natural map $\Sec{}{V/W} \to
  \varprojlim_n \Sec{}{V_n/W_n}$ is bijective. Indeed, a section $t
  \colon W \to V$ of $f$ is equivalent to a closed immersion $W \to
  W\times_S V$ such that the composition with the first projection is
  the identity. The algebraic $S$-stack $W\times_S V$ is also
  separated and \cite[4.6]{conradfmlgaga}
  now gives the claim. To see
  that the functor $\Psi_{\pi,I}$ is fully faithful, we simply 
  observe that if $Z$, $\widetilde{Z} \in \qfc{X}{S}$, then 
  \[
  \Hom_{\qfc{X}{S}}(Z,\widetilde{Z}) = \Sec{}{[Z\times_X \widetilde{Z}]/Z}.\qedhere
  \]
\end{proof}
Proving that $\Psi_{\pi,I}$ is an equivalence when $\pi$ is not
necessarily separated is the main technical contribution of this paper. In this
section, we will prove
\begin{thm}\label{thm:exthm}
  Let $\pi \colon X \to S$ be a morphism of algebraic stacks that is
  locally of finite type with quasi-compact and separated
  diagonal and affine stabilizers. Then 
  $\Psi_{\pi,I}$ is an equivalence.   
\end{thm}
Some interesting special cases of stacks with affine stabilizers are:
\begin{enumerate}
\item global quotient stacks,
\item stacks of global type \cite[Defn.~2.1]{rydh-2009}, 
\item stacks with quasi-finite diagonal,
\item algebraic spaces, and
\item schemes.
\end{enumerate}
We wish to point out that Theorem \ref{thm:exthm} is even new for
schemes. As outlined in \S\ref{subsec:outline}, we will prove Theorem
\ref{thm:exthm} by a d\'evissage on the non-abelian category
$\qfc{X}{S}$. This d\'evissage is combined with the Raynaud--Gruson Chow
Lemma \cite[Cor.~5.7.13]{MR0308104} to reduce to proving that
$\Psi_{\pi,I}$ is essentially surjective for a very special
class of morphisms of schemes $\pi \colon X \to S$. Before we get into our
main lemmas to set up the d\'evissage, the following definitions will be useful. 

Let $\pi \colon X \to S$ be a morphism of algebraic stacks that is locally
of finite type with quasi-compact and separated diagonal. We say that $(Z_n \to X_n)_{n\geq 0}\in \qfc{\cmpl{X}}{S}$ is
\textbf{effectivizable} if it lies in the essential image of
$\Psi_{\pi,I}$. If $(Z_n \to X_n)_{n\geq 0} \in
\qfc{\cmpl{X}}{S}$ is effectivizable, we say that an object $(Z \to X) \in
\qfc{X}{S}$ together with an isomorphism $e \colon \Psi_{\pi,I}(Z \to X)
\to (Z_n\to X_n)_{n\geq 0}$ is an \textbf{effectivization} of $(Z_n
\to X_n)_{n\geq 0}$. Note that given two effectivizations $((Z\to X),e)$,
$((Z'\to X),e')$ of
an effectivizable $(Z_n \to X_n)_{n\geq 0}$, then there exists a
unique isomorphism $\alpha \colon (Z' \to X) \to (Z \to X)$ such that
$e \circ \Psi_{\pi,I}(\alpha) = e'$ (Lemma \ref{lem:fullyfaithful}).

Let $(\varphi_n)_{n\geq 0} \colon (Z'_n \to X_n)_{n\geq 0}
\to (Z_n \to X_n)_{n\geq 0}$ be a morphism in
$\qfc{\cmpl{X}}{S}$. Let $J \subseteq \Orb_X$  be a coherent
sheaf of ideals. If $X$ is a scheme, then $(\varphi_n)$ corresponds to
a finite morphism of locally noetherian formal schemes
$\varphi \colon \fml{Z}' \to \fml{Z}$. Let
$\varphi^\sharp\colon\Orb_{\fml{Z}}\to\varphi_*\Orb_{\fml{Z}'}$ denote the
corresponding homomorphism. We say that $\varphi$ is
\textbf{$J$-admissible} if $J_{\fml{Z}} \cap \ker
\varphi^\sharp = 0$ and $\coker 
\varphi^\sharp$ is annihilated by $J_{\fml{Z}}$. The last condition is
equivalent to $\varphi_* J_{\fml{Z}'}\subseteq \im \varphi^\sharp$.
We say that
$\varphi$ is \textbf{strongly $J$-admissible} if $J_{\fml{Z}} \cap \ker
\varphi^\sharp = 0$ and $\varphi_*J_{\fml{Z}'}=\varphi^\sharp(J_{\fml{Z}})$.
If $X$ is any algebraic stack, locally of finite type
over $S$, then we say that $(\varphi_n)_{n\geq 0}$ is
$J$-admissible (resp.\ strongly $J$-admissible) if there exists a smooth
surjection from a
scheme $V \to X$ such that the adic system $(\varphi_n)_V : Z'_{V_n} \to Z_{V_n}$
induces a $J$-admissible (resp.\ strongly
$J$-admissible) morphism $\varphi_V\colon \fml{Z}'_V \to \fml{Z}_V$ of formal
schemes.
This notion does not depend on $V\to X$.

Our d\'evissage methods hinge on the
following two lemmas. Our first Lemma forms
the analogue of Step \itemref{dev:step_bir} given in
\S\ref{subsec:outline}. It is here that we first utilize Corollary \ref{cor:fml_pushouts_commute_stks}.
\begin{lem}\label{lem:jadm_eff}
  Let $\pi \colon X \to S$ be a morphism of algebraic stacks that is
  locally of finite type with quasi-compact and separated diagonal. Let
  $J \subseteq \Orb_X$ be a coherent ideal and let
  $(\varphi_n)_{n\geq 0} \colon (Z'_n \to X_n)_{n\geq 0} \to (Z_n \to
  X_n)_{n\geq 0}$ be a $J$-admissible morphism in
  $\qfc{\cmpl{X}}{S}$. Suppose that $(Z'_n \to X_n)_{n\geq 0}$ and
  $(Z_n\times_X V(J) \to X_n)_{n\geq 0}$ are effectivizable. Then $(Z_n
  \to X_n)_{n\geq 0}$ is effectivizable. 
\end{lem}
\begin{proof}
  There exists a cocartesian diagram in $\qfs{\cmpl{X}}$ (Corollary \ref{cor:fml_pushouts_commute_stks}):
  \[
  \xymatrix{(Z'_n \times_X V(J) \to X_n)_{n\geq 0} \ar@{^(->}[r]
    \ar[d]  & (Z'_n \to X_n)_{n\geq 0} \ar[d]^{(\varphi_n)_{n\geq 0}} \\ (Z_n \times_X V(J) \to X_n)_{n\geq 0}
    \ar@{^(->}[r]  & (W_n \to X_n)_{n\geq 0},} 
  \]
  where $(W_n \to X_n)_{n\geq 0} \in \qfc{X}{S}$ is
  effectivizable to $(W \to X)$. The universal properties show 
  that there is a uniquely induced morphism $(b_n)_{n\geq 0} \colon (W_n
  \to X_n)_{n\geq 0} \to
  (Z_n \to X_n)_{n\geq 0}$ such that each $b_n$ is finite. We will show
  that $(b_n)$ is strongly $J$-admissible and that if $(\varphi_n)$ is strongly
  $J$-admissible, then $(b_n)$ is an isomorphism. It follows that $(Z_n)$ is
  the pushout of an effectivizable diagram and is thus effectivizable.

  That $(b_n)$ is strongly $J$-admissible (resp.\ an isomorphism) can be
  verified smooth-locally on $X$. Thus after replacing $X$
  by an affine and noetherian scheme, we may pass from the adic
  systems above to noetherian formal schemes. That is,
  $(\varphi_n)_{n\geq 0}$ is given by a finite morphism of noetherian formal
  schemes $\varphi \colon (\fml{Z'} \to \cmpl{X}) \to (\fml{Z} \to
  \cmpl{X})$.

  By assumption, $\varphi$ is $J$-admissible (resp.\ strongly
  $J$-admissible), so
  we also know that $J_{\fml{Z}} \cap \ker \varphi^\sharp = 0$ and
  that $\varphi_* J_{\fml{Z}'}\subseteq \im \varphi^\sharp$
  (resp.\ $\varphi_* J_{\fml{Z}'}=\varphi^\sharp(J_{\fml{Z}})$). Moreover,
  $(b_n)_{n\geq 0}$ is given by a finite morphism of 
  noetherian formal schemes $\beta \colon \fml{W} \to \fml{Z}$.

  That $\beta$ is strongly $J$-admissible (resp.\ an isomorphism) is Zariski
  local on $\fml{Z}$.
  Thus, we may assume that $\fml{Z} = \spf_{IA} A$, $\fml{Z}' =
  \spf_{IA'} A'$, and that $\varphi$ is given by a finite morphism $f \colon A
  \to A'$. Let $B=A/JA \times_{A'/JA'} A'$. By Theorem
  \ref{thm:fml_pushouts_commute}, there is a natural isomorphism
  $\fml{W} \cong \spf_{IA} B$.
  It remains to prove that the natural map $A \to B$
  is strongly $J$-admissible (resp.\ an isomorphism).
  To this end we form the
  commutative diagram with exact rows: 
\[
\xymatrix@C+5mm{0 \ar[r] & JA \ar[r]^-{m\mapsto (m,0)} \ar[d]^f & \ar[d]^{d_1}
  A \times A' \ar[r] & A/JA \times A'\ar[d]^{d_2}
  \ar[r] & 0\\
  0 \ar[r] & JA' \ar[r] & A' \ar[r] & A'/JA' \ar[r] & 0,}
\]
where
$d_1(a,a') = f(a) - a'$ and $d_2(\bar{a},a') =
\bar{f}(\bar{a}) - \bar{a'}$.
By the Snake Lemma, there is an exact sequence of $A$-modules:
\[
\xymatrix{0 \ar[r] & JA \cap \ker(f) \ar[r] & A \ar[r] & B
  \ar[r] & JA'/f(JA) \ar[r] &0.}
\]
By hypothesis, $JA\cap \ker(f) = 0$ and $JA'\subseteq f(A)$
(resp.\ $JA'=f(JA)$). Thus, if $\varphi$ is strongly $J$-admissible, then
$A \to B$ is an isomorphism. If $\varphi$ is only
$J$-admissible, then $A\to B$ is injective and $JB\subseteq A$. Furthermore,
$JB\subseteq (0,JA')$ and $B/(0,JA')=A/JA$. It follows that $A/JA\to B/JB$ is
injective so that $JB=JA$ and $A\to B$ is strongly
$J$-admissible.
\end{proof}
In order to apply Lemma \ref{lem:jadm_eff} we have our next
Lemma. This is the analogue of \cite[III.5.3.4]{EGA} in
our situation; forming part of Step \itemref{dev:step_adj} in the outline
given in \S\ref{subsec:outline}. To
prove this Lemma, we must combine
the main techniques of the paper developed thus far---specifically
Corollary \ref{cor:fml_pushouts_commute_stks} and Theorem
\ref{thm:pf_sch_adj}.
\begin{lem}\label{lem:new_patch}
  Let $X$ be an algebraic stack of finite type over $S$ with
  quasi-compact and separated diagonal.
  Let $(Z_n \to X_n)_{n\geq 0} \in \qfc{\cmpl{X}}{S}$.
  Suppose that there is a proper and schematic morphism
  $p \colon X' \to X$ and an open substack $U \subseteq X$
  such that the morphism $p^{-1}(U) \to U$ is finite and
  flat.
  Assume, in addition, that $\cmpl{p}^*(Z_n \to X_n)_{n\geq 0} \in
  \qfc{\cmpl{X}'}{S}$ is effectivizable. Then there exists a coherent
  ideal $J \subseteq \Orb_X$, with $|\supp(\Orb_X/J)|$ equal to the
  complement of $U$ in $X$, and a $J$-admissible morphism $(\varphi_n)_{n\geq 0} \colon
  (\widetilde{Z}_n \to X_n)_{n\geq 0} \to (Z_n \to X_n)_{n\geq 0}$ with
  $(\widetilde{Z}_n \to X_n)_{n\geq 0}$ effectivizable.
\end{lem}
\begin{proof}
  Let $X'' = X'\times_X X'$. For $j=1$ and $2$ let $s_j \colon X'' \to X'$
  denote the $j$th projection. Let $q = p\circ s_1$ and observe that
  there is a $2$-morphism $\alpha \colon q \Rightarrow p \circ s_2$. By
  hypothesis, $\cmpl{p}^*(Z_n \to X_n)_{n\geq 0}$ admits an effectivization $(Z' \to X') \in \qfc{X'}{S}$. By Theorem \ref{thm:prop_push_fml}, for
  $j=1$ and $2$, there exist natural morphisms $\eta_j \colon (s_j)_!(s_j)^*Z' \to
  Z'$ in $\qfc{X'}{S}$. By Theorem \ref{thm:prop_push_fml}, there are also natural morphisms $p_!(\eta_j) \colon p_!(s_j)_!(s_j)^*Z' \to
  p_!Z'$ in $\qfc{X}{S}$. Moreover, there is a natural isomorphism of functors
  $p_!(s_1)_! \Rightarrow q_!$. The $2$-morphism $\alpha$ also induces
  a natural isomorphism of functors $p_!(s_2) \Rightarrow q_!$ and an
  isomorphism $\Psi_{\pi\circ q,I}(s_1^*Z'_1) \cong \Psi_{\pi \circ
    q,I}(s_2^*Z'_2)$ in $\qfc{\cmpl{X}''}{S}$, where $\pi \colon X \to S$
  denotes the structure morphism of $X$. By Lemma
  \ref{lem:fullyfaithful}, we obtain an isomorphism $\nu \colon s_1^*Z'_1
  \to s_2^*Z'_2$ in $\qfc{X''}{S}$. Let $Z''= s_1^*Z'$. Combining the afforementioned isomorphisms, we obtain for $j=1$ and $2$ natural
  morphisms $t_j \colon q_!Z'' \to p_!Z'$ in $\qfc{X}{S}$. The morphisms $t_j$ are
  finite, thus the coequalizer diagram $[q_!Z'' \rightrightarrows
  p_!Z']$ has a colimit, $\widetilde{Z}$ in $\qfc{X}{S}$ (Theorem
  \ref{thm:pushouts}). By Theorem \ref{thm:prop_push_fml}
  there are natural isomorphisms:
  \begin{align*}
    \Psi_{\pi,I}(p_!Z') &\cong \cmpl{p}_!(Z'\times_{X'} X'_n \to
    X'_n)_{n\geq 0} \cong \cmpl{p}_!\cmpl{p}^*(Z_n \to X_n)_{n\geq
      0}\\
    \Psi_{\pi,I}(q_!Z') &\cong \cmpl{q}_!(Z''\times_{X''} X''_n \to
    X''_n)_{n\geq 0} \cong \cmpl{q}_!\cmpl{q}^*(Z_n \to X_n)_{n\geq 0}.
  \end{align*}
  Moreover, the morphism $\cmpl{p}_!\cmpl{p}^*(Z_n \to X_n)_{n\geq
    0} \to (Z_n \to X_n)_{n\geq 0}$ coequalizes the two morphisms
  $\cmpl{q}_!\cmpl{q}^*(Z_n\to X_n)_{n\geq 0} \rightrightarrows
  \cmpl{p}_!\cmpl{p}^*(Z_n \to X_n)_{n\geq 0}$. By Corollary
  \ref{cor:fml_pushouts_commute_stks}, there is a uniquely induced
  morphism $(\varphi_n)_{n\geq 0} \colon (\widetilde{Z}\times_X X_n \to X_n)_{n\geq
    0} \to (Z_n \to X_n)_{n\geq 0}$ in $\qfc{\cmpl{X}}{S}$. 

  It remains to prove that $(\varphi_n)_{n\geq 0}$ is $J$-admissible for a
  suitable ideal $J$. First,
  let $J\subseteq \Orb_X$ be any coherent ideal defining the complement of $U$.
  We will show that $(\varphi_n)_{n\geq 0}$ is $J^N$-admissible for sufficiently
  large $N$. Note
  that the verification of this is local on $X$ for the smooth
  topology. Indeed, by Corollary \ref{cor:fml_pushouts_commute_stks}
  and Theorem \ref{thm:prop_push_fml}, the construction of
  $(\varphi_n)_{n\geq 0}$ is compatible with smooth base change on $X$. So,
  we may henceforth assume that $X=\spec B$, where $B$ is an
  $R$-algebra of finite type, and that $(Z_n \to
  X_n)_{n\geq 0}$ is given by a morphism of locally noetherian formal
  schemes $(\fml{Z} \to \cmpl{X})$. We next observe that we may work
  Zariski locally on $\fml{Z}$. To see this, we note that the morphism
  $\cmpl{p}_!\cmpl{p}^*\fml{Z} \to \fml{Z}$ is just the Stein
  factorization of $\cmpl{p}^*\fml{Z} \to \fml{Z}$ (Lemma
  \ref{lem:uni_gs}) and this is compatible with flat base change on
  $\fml{Z}$ (Lemma \ref{lem:st_bc_comp}), and similarly for the
  morphisms $\cmpl{q}_!\cmpl{q}^*\fml{Z} \rightrightarrows
  \cmpl{p}_!\cmpl{p}^*\fml{Z}$. Also, let $\widetilde{Z}^{\wedge} =
  \widetilde{\fml{Z}}$ and let $(\varphi_n)_{n\geq 0}$ be given by the finite
  morphism $\varphi \colon \fml{\widetilde{Z}} \to \fml{Z}$. Then Theorem
  \ref{thm:fml_pushouts_commute} implies that $\Orb_{\widetilde{\fml{Z}}} =
  \ker(\Orb_{\cmpl{p}_!\cmpl{p}^*\fml{Z}} \xrightarrow{\cmpl{t}^\sharp_1 -
    \cmpl{t}^\sharp_2} \Orb_{\cmpl{q}_!\cmpl{q}^*\fml{Z}})$. Thus the
  formation of $\widetilde{\fml{Z}}$ is also compatible with flat 
  base change on $\fml{Z}$. Consequently, we may assume that $\fml{Z}
  = \spf_{IA} A$, for some $B$-algebra $A$ that is $IA$-adically
  complete. Let $Z = \spec A$, then the morphism $\cmpl{Z} \to
  \cmpl{X}$ induces a morphism of schemes $Z \to X$. 

  Let $p_Z \colon Z' \to Z$ (resp.~$q_Z \colon Z'' \to Z$) denote the
  pullback of $p$ (resp.~$q$) along $Z \to X$. Let $A' =
  \Gamma(Z',\Orb_{Z'})$ and $A'' = \Gamma(Z'',\Orb_{Z''})$
  (which are both coherent $A$-modules because $p_Z$ and $q_Z$ are
  proper). Then there are two induced morphisms $d_1$, $d_2 \colon A' \to
  A''$ and $\widetilde{A} := \Gamma(\widetilde{Z},\Orb_{\widetilde{Z}}) = \ker(A'
  \xrightarrow{d_1-d_2} A'')$. The morphism $\varphi \colon \fml{\widetilde{Z}}
  \to \fml{Z}$ induces a natural finite map $f \colon A \to \widetilde{A}$.

  It remains to show that we can find an
  $N\gg 0$ such that $(J^NA) \cap \ker f = 0$ and $J^NA$
  annihilates $\coker f$. Now, on the complement, $U_Z$, of the support of
  $A/JA$ on $\spec A$, the morphism $p_Z^{-1}(U_Z) \to U_Z$ is finite and
  flat. So, if $g\in JA$, then $A''_g = A'_g \tensor_{A_g}
  A'_g$. By finite flat descent, $\varphi_g \colon A_g \to
  \widetilde{A}_g$ is an isomorphism. Since $A$ is noetherian, we deduce
  immediately that the kernel and cokernel of $\varphi$ are
  annihilated by $J^NA$ for some $N\gg 0$. By the Artin--Rees Lemma
  \cite[0\textsubscript{I}.7.3.2.1]{EGA}, and by possibly increasing $N$, we
  can ensure that $(J^NA) \cap \ker f = 0$. The result follows.
\end{proof}
We may finally prove Theorem \ref{thm:exthm}.
\begin{proof}[Proof of Theorem \ref{thm:exthm}]
By Lemma \ref{lem:fullyfaithful}, it remains to show that
the functor ${\Psi}_{\pi,I}$ is essentially
surjective. We divide the proof of this into a number of cases.

\noindent{\bf Basic case.} Let $\pi \colon X \to S$ be a morphism of schemes that is of finite type and factors as $X
\xrightarrow{f} P \xrightarrow{g} S$, 
where  $f$ is \'etale and $g$ is projective. Let $(Z_n \to 
X_n)_{n\geq 0} \in \qfc{\cmpl{X}}{S}$, which we may regard as a
quasi-finite and separated morphism of locally noetherian formal
schemes $(\fml{Z} \to \cmpl{X})$. Then the composition
$\fml{Z} \to \cmpl{X} \to \cmpl{P}$ is
quasi-finite and proper, 
  since $P$ is separated and $\fml{Z}$ is $S$-proper. Hence,
  the morphism $\fml{Z} \to \cmpl{P}$ is finite
  \cite[III.4.8.11]{EGA}. As $g
  \colon P \to S$ is projective, there is a finite $P$-scheme  
  $Z$ such that
  $(\cmpl{Z} \to \cmpl{P}) \cong (\fml{Z} \to \cmpl{P})$
  in $\qfc{\cmpl{P}}{S}$  
  \cite[III.5.4.4]{EGA}. The morphism $f$ is \'etale, thus Corollary
  \ref{cor:sections_proper_bc} implies that there is a unique $P$-morphism
  $Z \to X$ lifting the given $P_0$-morphism $Z_0 \to X_0$. Thus we have
  obtained $(Z \to X) \in \qfc{X}{S}$ with completion isomorphic to
  $\fml{Z}$ over $\cmpl{X}$. We conclude that the functor $\Psi_{\pi,I}$ is an
  equivalence of categories in this case.

\noindent{\bf Quasi-compact with quasi-finite and separated diagonal
  case.} Let $\pi 
  \colon X \to S$ be a morphism of algebraic stacks that is of finite type with quasi-finite and separated
  diagonal. We now prove that the functor $\Psi_{\pi,I}$ is an
  equivalence by noetherian induction on the closed subsets of
  $|X|$. Thus, for any closed immersion $i \colon V \hookrightarrow X$ with
  $|V| \subsetneq |X|$, we may assume that the functor $\Psi_{\pi\circ
    i,I} \colon \qfc{V}{S} \to \qfc{\cmpl{V}}{S}$ is an equivalence. 

  Let $(Z_n \to X_n)_{n\geq 0} \in 
  \qfc{X}{S}$. Then there exists a
  $2$-commutative diagram:
  \[
  \xymatrix@-0.8pc{X \ar[d]_\pi & \ar[l]_p X' \ar[d]^f \\ S & \ar[l]_g P,}
  \]
  such that $g$ is projective, $f$ is \'etale, $X'$ is a scheme, and
  $p$ is proper, schematic, surjective, and over a dense open subset
  $U \subseteq X$ the morphism $p^{-1}(U) \to U$ is finite and
  flat \cite[Thm.~8.9]{rydh-2009}. By the Basic Case considered above, $\cmpl{p}^*(Z_n \to
  X_n)_{n\geq 0} \in \qfc{\cmpl{X}'}{S}$ is effectivizable. By Lemma
  \ref{lem:new_patch}, there exists a coherent ideal $J \subseteq
  \Orb_X$, with $|\supp(\Orb_X/J)|$ equal to the complement of $U$ in
  $X$, and a $J$-admissible morphism $(\varphi_n)_{n\geq 0} \colon (\widetilde{Z}_n
  \to X_n)_{n\geq 0} \to (Z_n \to X_n)_{n\geq 0}$ with $(\widetilde{Z}_n
  \to X_n)_{n\geq 0}$ effectivizable. By noetherian induction, 
  $(Z_n\times_X V(J) \to X_n)_{n\geq 0} \in \qfc{\cmpl{X}}{S}$ is
  effectivizable. By Lemma \ref{lem:jadm_eff}, $(Z_n \to X_n)_{n\geq
    0}$ is effectivizable.

  \noindent{\bf General case.} We now assume that $\pi \colon X \to S$ is as in the statement of the Theorem. It remains to prove that the functor $\Psi_{\pi,I}$ is essentially surjective. To show this, we first reduce to the case where $\pi$ is quasi-compact. Let $(Z_n \to X_n)_{n\geq 0} \in
  \qfc{\cmpl{X}}{S}$.  Let
  ${O}_X$ denote the set of quasi-compact open substacks of $X$. The set
  $O_X$ is ordered by inclusion, and we note that $\{ |W \times_X
  Z_0|\}_{W \in {O}_X}$ is an open cover of the quasi-compact
  topological space $|Z_0|$. Thus, there is a $W \in {O}_X$ such that
  the canonical $X$-morphism $W\times_X Z_0 \to Z_0$ is an
  isomorphism. In particular, the map $Z_0 \to X$
  factors uniquely as $Z_0 \to W \hookrightarrow X$. Since the open
  immersion $W\hookrightarrow X$ is \'etale, it follows that the maps
  $Z_n \to X$ also factor compatibly through $W$. We can thus replace
  $X$ with $W$ and assume henceforth that $X$ is quasi-compact. 

  Next, we note that there is an open substack $X^{\mathrm{qf}}
  \subseteq X$ with the property that a point $x$ of
  $X$ belongs to $X^{\mathrm{qf}}$ if and only if the stabilizer at
  $x$ is finite. Indeed, by Chevalley's 
  Theorem \cite[IV.13.1.4]{EGA}, the locus of points in the inertia
  stack $I_{X/S} \to X$ that are isolated in their fibers is an 
  open subset of $I_{X/S}$. Intersecting this open set with the identity section $e \colon X \to
  I_{X/S}$ gives $X^{\mathrm{qf}}$ as $I_{X/S} \to X$ is a
  relative group stack (the dimension of a finite type group algebraic
  space over a field is its local dimension at the identity). 

  Arguing as before and applying the previous case, it 
  now remains to show that the map $Z_0 \to X$ factors through 
  $X^{\mathrm{qf}}$. So we let
  $I_X \to X$ (resp.~$I_{Z_0} \to Z_0$) denote the inertia stack of
  $X$ (resp.~$Z_0$) over $S$. The morphism $Z_0 \to X$ is separated and
  representable thus $I_{Z_0} \to I_X\times_{X} Z_0$ is a closed
  immersion over $Z_0$. By assumption, $I_X \to X$ has affine fibers
  and $I_{Z_0} \to Z_0$ is proper. Hence the fibers of $I_{Z_0}
  \to Z_0$ are proper and affine, thus finite. By Zariski's Main Theorem
  \cite[Cor.~A.2.1]{MR1771927}, $I_{Z_0} \to Z_0$ is finite, and so $Z_0$
  has finite diagonal over $S$ (because $Z_0$ is separated over
  $S$). We now form the $2$-commutative diagram:
  \[
  \xymatrix@-0.8pc{Z_0\times_X Z_0 \ar[r] \ar@/_1pc/[dr] & Z_0\times_S Z_0
    \ar[d] \\ & Z_0.} 
  \]
  The morphism $Z_0 \to X$ is also quasi-finite, thus the morphism
  $Z_0\times_X Z_0 \to Z_0$ is quasi-finite. The morphism $Z_0\times_S
  Z_0 \to Z_0$ has finite diagonal, thus $Z_0\times_X
  Z_0 \to Z_0\times_S Z_0$ is quasi-finite. We may now conclude that
  if $z$ is a point of $Z_0$, then its image in
  $X$ has finite stabilizer. Thus $Z_0$ factors
  through $X^{\mathrm{qf}}$ and we deduce the result. 
\end{proof}
\begin{rem}
It is possible that Theorem~\ref{thm:exthm} extends to any stack $X$ with
locally quasi-finite diagonal, that is, without requiring that $\Delta_X$ is
separated and quasi-compact. For such $X$, we consider the full subcategory
$\lqfc{\cmpl{X}_{/|X_0|}}{S}\subseteq \lqf{\cmpl{X}_{/|X_0|}}{S}$ consisting of
systems of locally quasi-finite representable morphisms $(Z_n\to X)_{n\geq 0}$
such that the composition $Z_0 \to X \to S$ is proper.

The generalization of Lemma~\ref{lem:fullyfaithful}, that $\Psi_{\pi,I}\colon
\lqfc{X}{S} \longrightarrow \lqfc{\cmpl{X}}{S}$ is fully faithful, follows as
for stacks with separated diagonals, but using Theorem~\ref{thm:exthm} instead
of \cite[Thm.\ 1.4]{MR2183251}.
The corresponding generalizations of
Lemmas~\ref{lem:jadm_eff} and~\ref{lem:new_patch} would follow if we extend the
results of Sections~\ref{ch:coeq} and~\ref{ch:stein_factorizations} to
include certain non-separated quasi-finite morphisms. This is probably not too
difficult to
accomplish, although one may have to develop a theory of non-separated formal
algebraic spaces.

Finally, to obtain the generalization of Theorem~\ref{thm:exthm} we would also
need a suitable Chow lemma for stacks with locally quasi-finite diagonals. It
would suffice to find a proper generically finite and flat morphism $p\colon
X'\to X$ together with an \'etale, not necessarily representable, morphism
$f\colon X'\to P$ with $P\to S$ projective. Whether such a proper covering
exists is not clear to the authors.
\end{rem}


\section{Algebraicity of the Hilbert stack}\label{ch:hilbstk}
Before we get to the proof Theorem \ref{thm:main}, we will require
an analysis of some spaces of sections. Let $T$ be a scheme and let $s \colon 
V' \to V$ be a representable morphism of algebraic $T$-stacks. Define
$\SEC{T}{V'/V}$ to be the sheaf that takes a $T$-scheme $W$ to the
set of sections of the morphism $V'\times_T W\to V\times_T W$. We
require an improvement of \cite[Prop.\ 5.10]{MR2239345} and
\cite[Lem.\ 2.10]{MR2233719}, which we prove using \cite[Thm.\
D]{hallj_coho_bc}.
\begin{prop}\label{prop:sections_mor}
  Let $T$ be a scheme and let $p \colon Z \to T$ be a morphism of algebraic stacks that is proper, flat, and of finite presentation. Let $s \colon Q \to Z$  be a 
  quasi-finite, separated, finitely presented, and representable
  morphism. Then the $T$-sheaf $\SEC{T}{Q/Z}=p_*Q$ is
  represented by a quasi-affine $T$-scheme which is of finite presentation. 
\end{prop}
\begin{proof}
  By a standard limit argument~\cite[Prop.\ B.3]{rydh-2009}, we can assume that
  $T$ is noetherian.
  By Zariski's Main Theorem~\cite[Thm.\ 16.5(ii)]{MR1771927}, there is a
  finite morphism
  $\bar{Q} \to Z$ and an open immersion $Q\hookrightarrow \bar{Q}$ over $Z$.
  In particular, we see that there is a natural transformation of
  $T$-sheaves $\SEC{T}{Q/Z} \to \SEC{T}{\bar{Q}/Z}$ which is
  represented by open immersions. Hence, we may
  assume for the remainder that the morphism $s \colon Q\to Z$ is
  finite. Next, observe that $\SEC{T}{Q/Z}$ is a subfunctor of
  the $T$-presheaf
  $\underline{\Hom}_{\Orb_Z/T}(s_*\Orb_Q,\Orb_Z)$. By
  \cite[Thm.\ D]{hallj_coho_bc},
  the sheaf $\underline{\Hom}_{\Orb_Z/T}(s_*\Orb_Q,\Orb_Z)$ is represented by
  a scheme that is affine and of finite type over $T$. Similarly,
  $\underline{\Hom}_{\Orb_Z/T}(\Orb_Z,\Orb_Z)$ and
  $\underline{\Hom}_{\Orb_Z/T}(s_*\Orb_Q\otimes_{\Orb_Z}s_*\Orb_Q,\Orb_Z)$ are
  affine. It follows that
  $\SEC{T}{Q/Z}\hookrightarrow \underline{\Hom}_{\Orb_Z/T}(s_*\Orb_Q,\Orb_Z)$
  is represented by closed immersions and the result follows.
\end{proof}
\begin{cor}\label{cor:hom_qfmor}
  Let $T$ be a scheme and let $X \to T$ be a morphism of algebraic stacks that is locally of finite presentation. Suppose that we have quasi-finite, separated and
  representable morphisms $s_i \colon Z_i \to X$ for $i=1$, $2$ such that
  $Z_1$ and $Z_2$ are proper, flat and of finite presentation over $T$.
  Then the functor $\Hom_X(Z_1,Z_2)$ on $\SCH{T}$, given by
  $T'
  \mapsto \Hom_{X_{T'}}\bigl((Z_1)_{T'},(Z_2)_{T'}\bigr)$, is represented by
  a scheme that is quasi-affine over $T$.
  In particular, the 
  open subfunctor $\Isom_X(Z_1,Z_2) \subset
  \Hom_X(Z_1,Z_2)$ parameterizing isomorphisms is represented by
  a scheme that is quasi-affine over~$T$. 
\end{cor}
\begin{proof}
  Note that $\Hom_{X}(Z_1,Z_2) = \SEC{T}{(Z_1\times_X Z_2)/Z_1}$.
  Thus, by Proposition \ref{prop:sections_mor} the functor
  $\Hom_X(Z_1,Z_2)$ has the asserted properties. 
\end{proof}
\begin{proof}[Proof of Theorem \ref{thm:main}]
The results of \cite[\S9]{hallj_openness_coh}, together with Theorem
\ref{thm:exthm}, show that the stack $\HS{X/S}$ is algebraic and
locally of finite presentation over $S$. We now apply Corollary
\ref{cor:hom_qfmor} to obtain the asserted separation property. 
\end{proof}


\appendix
\section{Coherent cohomology of formal schemes}
In this appendix, we prove that cohomology commute with flat base change of
formal
schemes. This is well-known for schemes, but we could not find a
suitable reference for formal schemes. Note that in the adic case, this
result follows from the more general result \cite[Prop.~7.2(b)]{MR1716706}.
\begin{prop}\label{prop:fml_bc}
  Consider a cartesian diagram of locally noetherian formal schemes:
  \[
  \xymatrix@-0.8pc{\fml{X}' \ar[r]^{p'} \ar[d]_{\varpi'} & \fml{X}
    \ar[d]^{\varpi} \\ \fml{Y}' \ar[r]^p & \fml{Y},}
  \]
  where $\varpi$ is proper and $p$ is flat. Let $\fml{F}$ be a
  coherent $\Orb_{\fml{X}}$-module, then for any $q\geq 0$ the base
  change morphism $p^*R^q\varpi_*\fml{F} \to
  R^q\varpi'_*p'^*\fml{F}$ is a topological  
  isomorphism. 
\end{prop}
\begin{proof}
  By \cite[III.3.4.5.1]{EGA}, the statement is
  Zariski local on $\fml{Y}$ and $\fml{Y}'$ so we may assume that
  $\fml{Y} = \spf_I R$ for some $I$-adic noetherian ring $R$,
  $\fml{Y}' = \spf_{I'} R'$ for some $I'$-adic noetherian ring $R'$,
  $IR' \subseteq I'$, and $R \to R'$ is a flat morphism of rings. It
  suffices to prove that the morphism $H^q(\fml{X},\fml{F})\cmpl{\tensor}_R
  R' \to H^q(\fml{X}',p'^*\fml{F})$ is a topological isomorphism. Note
  that $p$ factors as $\spf_{I'} R' \xrightarrow{r} \spf_{IR'}R'
  \xrightarrow{q}  \spf_I R$. Since $q$ is adic, we are reduced to proving the
  Proposition when $p$ is also adic or when $R=R'$. Note that
  cohomology can be shown to commute with adic flat base change by minor
  modifications to the statements and arguments of
  \cite[0\textsubscript{III}.13.7.8]{EGA}, combined with
  \cite[0\textsubscript{III}.13.7.7]{EGA} and
  \cite[III.3.4.3]{EGA} to deal with flat base change instead of localizations.
  It remains to prove the Proposition when $R=R'$.

  For each $k$, $l\geq 0$, set $R_{k,l} = R/(I^{k+1},I'^{l+1})$, $R_k
  = R/I^{k+1}$, and $R'_l = R/I'^{l+1}$. If $k\geq l$, then $R_{k,l} =
  R'_{l}$ (because $I \subseteq I'$). Also, 
  for each fixed $k$, the noetherian ring $R_k$ is $I'$-adically
  complete, and the natural map $R_k \to \varprojlim_l R_{k,l}$ is a
  topological isomorphism. Let $X_{k,l} = \fml{X}\tensor_R R_{k,l}$,
  ${X}_k = \fml{X}\tensor_R R_{k}$ (which we view as noetherian schemes), $F_{k,l} =
  \fml{F}\tensor_R R_{k,l}$, and ${F}_k = \fml{F}{\tensor}_R
  R_k$. By \cite[III.4.1.7]{EGA}, for each $k\geq 0$, there are
  natural induced topological isomorphisms:
  \[
  H^q(X_k,{F}_k)^\wedge \cong \varprojlim_l H^q(X_{k,l},F_{k,l}),
  \]
  where $H^q(X_k,{F}_k)^\wedge$ denotes the completion of
  the $R_k$-module $H^q(X_k,{F}_k)$ with respect to the $I'$-adic
  topology. Note, however, that $X_k \to \spec R_k$ is proper, thus
  $H^q(X_k,{F}_k)$ is a coherent $R_k$-module
  \cite[III.3.2.1]{EGA}. Consequently, $H^q(X_k,{F}_k)$ is
  $I'$-adically complete \cite[0\textsubscript{I}.7.3.6]{EGA}, and we have
  topological
  isomorphisms: 
  \[
  H^q(X_k,{F}_k) \cong \varprojlim_l H^q(X_{k,l},F_{k,l}).
  \]
  Let $X'_l = \fml{X}\tensor_R R'_l$ and $F_l'= \fml{F}\tensor_R
  R'_l$. Applying the functor $\varprojlim_k$ to both sides of the
  isomorphism above we obtain natural isomorphisms of $R'$-modules:
  \[
  \varprojlim_k H^q(X_k,{F}_k) \cong \varprojlim_k \varprojlim_l
  H^q(X_{k,l},F_{k,l}) \cong \varprojlim_l \varprojlim_k
  H^q(X_{k,l},F_{k,l}) \cong \varprojlim_l H^q(X'_l,F_l').
  \]
  By \cite[III.3.4.4]{EGA}, we have naturally induced topological
  isomorphisms: 
  \[
  H^q(\fml{X},\fml{F}) \to   \varprojlim_k H^q(X_k,{F}_k) \quad
  \mbox{and}\quad H^q(\fml{X}',\fml{F}') \to   \varprojlim_l H^q(X'_l,{F}'_l),
  \]
  and we deduce the result.
\end{proof}
As a Corollary, we can prove Zariski's Connectedness Theorem for
formal schemes, using the same argument as \cite[III.4.3.2 and 4.3.4]{EGA}.
\begin{cor}\label{cor:zar_con_fml}
  Let $\varpi \colon \fml{X} \to \fml{Y}$ be a proper and Stein morphism
  of locally noetherian formal schemes. Then $\varpi$ has
  geometrically connected fibers. 
\end{cor}
\begin{proof}
  Use \cite[0\textsubscript{III}.10.3.1]{EGA} and Proposition \ref{prop:fml_bc}
  to reduce to $\fml{Y}=\spf_{IR} R$ where $R$ is a complete local ring with
  maximal ideal $I$ and algebraically closed residue field $R/I$. Then
  $|\fml{X}|$ equals the unique fiber of $\varpi$. Finally observe that
  $\varpi_*\Orb_{\fml{X}}=\Orb_{\fml{Y}}$ is local so
  $|\fml{X}|$ is connected.
\end{proof}


\section{Henselian pairs}
In this appendix, we address some simple results about \'etale morphisms that we
could not locate in the literature. As an application, we 
strengthen a well-known result about the stability of henselian pairs
under proper morphisms. 

For an algebraic space $S$, let $\Et{S}$ denote the category of
morphisms of algebraic spaces $V \to S$ that are \'etale. Let
$\et{S}$ denote the small \'etale site of $S$: this has underlying
category $\Et{S}$ and the coverings are jointly surjective families of
morphisms. It is well-known that the natural functor:
\[
\Et{S} \to \SH{\et{S}},\quad (V \to S) \mapsto \Hom_S(-,V)
\]
is an equivalence of categories \cite[V.1.5]{MR559531}, where
$\SH{\et{S}}$ denotes category of $\et{S}$-sheaves of \emph{sets} (even
if $S$ is a scheme, we need to allow $V$ to be an algebraic space).
If $f \colon X \to S$ is a morphism of algebraic spaces, there are thus
natural functors $f^* \colon \Et{S} \to \Et{X}$ and $f_* \colon \Et{X} \to
\Et{S}$, with $f^*$ left adjoint to $f_*$. We define
$\Etad{X}{S}$ to be the full subcategory of $\Et{X}$ with objects $W
\to X$ such that for every geometric point $\bar{s}$ of $S$ and every
connected component $Z$ of $W_{\bar{s}}$, the induced morphism $Z \to
X_{\bar{s}}$ is an open and closed immersion. We now have the main
technical result of this appendix. 
\begin{prop}\label{prop:et_prop_geom_conn}
  Let $f \colon X \to S$ be a proper and surjective morphism of schemes
  with geometrically connected fibers. Then the functor $f^* \colon \Et{S}
  \to \Et{X}$ is fully faithful with image $\Etad{X}{S}$.
\end{prop}
\begin{proof}
  To show that $f^*$ is fully faithful, it is
  sufficient to prove that if $V \in \Et{S}$, then the natural
  map $V \to f_*f^*V$ is an isomorphism. This may be verified over the
  geometric points of $\bar{s}$ of $S$. Note, however, that the functor $f^*$ is  
  trivially compatible with arbitrary base change on $S$. Since $f$ is
  proper, it is a basic case of the Proper Base Change Theorem
  \cite[XII.5.1(i)]{SGA4} that $f_*$ is also compatible with arbitrary base
  change on $S$. Thus, it suffices to prove the result where $S=\spec
  k$ with $k$ an algebraically closed field. Since $f$ is surjective,
  this is trivial.

  To classify the image of $f^*$, it remains to show that if $W\in
  \Etad{X}{S}$, then the natural map $f^*f_*W \to W$ is an
  isomorphism. This may be verified at geometric points $\bar{x}$ of
  $X$. The remarks above about base change also apply here, so we are
  again reduced to the situation where $S=\spec k$ with $k$ an
  algebraically closed field. In this case, the connected components
  of $W$ are all open and closed subschemes of $X$. Since $X$ is
  geometrically connected, we conclude that $W = \amalg_{w\in
    \pi_0(W)} X$. It remains to show that the natural map
  $\amalg_{w\in \pi_0(W)} S \to f_*W$ is an isomorphism. This may be
  checked on global sections (since $k$ is algebraically closed)
  and there we have a natural bijection $\pi_0(W) \to \Hom_X(X,W)$.
\end{proof}
\begin{rem}\label{rem:et_sep_clopen}
  Let $f \colon X \to S$ be a proper morphism of schemes with
  geometrically connected fibers and let $\gamma \colon V_1 \to V_2$ be a morphism in
  $\Etad{X}{S}$. Then $\gamma$ is an open immersion (resp.\ an open and closed
  immersion, resp.\ separated) if and only if the induced
  morphism $f_*\gamma \colon f_*V_1 \to f_*V_2$ is such. The first is an
  easy consequence of the fact that open immersions are categorical
  monomorphisms in $\Et{X}$, and since $f_*$ preserves products, it
  preserves monomorphisms. The other two cases are even easier. 
\end{rem}
For a scheme $S$, let $\Of{S}$ denote its set of open and
closed subsets. Note that $\Of{X} \subseteq \Etad{X}{S}$. A
\emph{henselian pair} $(S,S_0)$ consists of a scheme $S$ and a closed 
immersion $S_0 \hookrightarrow S$ such that for any finite morphism $g \colon S' \to S$, the
natural map $\Of{S'} \to \Of{S'\times_S S_0}$ is bijective
\cite[IV.18.5.5]{EGA}. 
\begin{ex}\label{ex:hens_pair}
  Note that if $B$ is a noetherian ring, separated and complete for
  the topology defined by an ideal $I\subseteq B$, then $(\spec
  B,\spec B/I)$ is a henselian pair \cite[IV.18.5.16(ii)]{EGA}.
\end{ex}
We now obtain the following improvement of \cite[XII.5.5]{SGA4} and
\cite[IV.18.5.19]{EGA}, where it is proved for henselian pairs of the
form $(\spec A,\{\mathfrak{m}\})$, where $A$ is a henselian local ring and
$\mathfrak{m}$ is the maximal ideal of $A$. 
\begin{cor}\label{cor:hens_pair_prop}
  Let $(S,S_0)$ be a henselian pair and let $f \colon X \to S$ be a proper
  morphism of noetherian schemes. Then
  $(X,X\times_S S_0)$ is a henselian pair.
\end{cor}
\begin{proof}
  Let $X_0 = X\times_S S_0$ and take $f_0 \colon X_0 \to S_0$ to be the
  induced morphism. It is sufficient to prove that $\Of{X}
  \to \Of{X_0}$ is bijective. If $X \to \widetilde{X} \to S$ denotes the Stein 
  factorization of $f$, then Proposition \ref{prop:et_prop_geom_conn} and
  Remark \ref{rem:et_sep_clopen} implies
  that we have a bijection $\Of{\widetilde{X}} \to \Of{X}$. Since
  $\widetilde{X} \to S$ is finite, $(\widetilde{X},\widetilde{X}\times_S S_0)$ is
  a henselian pair \cite[IV.18.5.6]{EGA}. A similar analysis applies
  to the Stein factorization of $f_0$, $X_0 \to \widetilde{X}_0
  \xrightarrow{\tilde{f}_0} S_0$, 
  where we also obtain a bijection $\Of{\widetilde{X}_0} \to
  \Of{X_0}$. Denote by $g \colon X_0 \to \widetilde{X}\times_S S_0$ and
  $\tilde{g} \colon \widetilde{X}\times_S S_0 \to S_0$ the induced
  morphisms. There are now natural maps of coherent $\Orb_{S_0}$-algebras:  
  \[
  \tilde{g}_*\Orb_{\widetilde{X}\times_S S_0} \to \tilde{g}_*g_*\Orb_{X_0}
  \cong (f_0)_*\Orb_{X_0} \cong (\tilde{f}_0)_*\Orb_{\widetilde{X}_0}.
  \]
  Whence we obtain a finite $S_0$-morphism $h \colon \widetilde{X}_0 \to
  \widetilde{X}\times_S S_0$. But the morphisms $X_0 \to \widetilde{X}_0$ and
  $X_0 \to \widetilde{X}\times_S S_0$ are both surjective with
  geometrically connected fibers, thus $h$ also is surjective with geometrically
  connected fibers. Consequently, $h$ is a universal homeomorphism
  \cite[IV.18.12.11]{EGA} so that
  $\Of{\widetilde{X}\times_S S_0}\to \Of{\widetilde{X}_0}$ is bijective.
  The result follows. 
\end{proof}
Note that one of the strengths of henselian pairs is that they enable
the computation of sections of sheaves. Indeed, if $(S,S_0)$ is a henselian
pair, with $S$ quasi-compact and quasi-separated, and $V \in \Et{S}$, then the
natural map $\Hom_S(S,V) \to \Hom_{S_0}(S_0,V\times_S S_0)$ is
bijective \cite[XII.6.5(i)]{SGA4}.
\begin{cor}\label{cor:sections_proper_bc}
  Let $(S,S_0)$ be a henselian pair where $S$ is a noetherian scheme. Fix a
  scheme $Y$ over $S$. Let $Z$ and $X$ be schemes
  over $Y$ such that $X\to Y$ is
  \'etale and $Z \to S$ is proper. Then the natural map:
  \[
  \Hom_Y(Z,X) \to \Hom_Y(Z\times_S S_0,X)
  \]
  is bijective. 
\end{cor}
\begin{proof}
The map in the statement is identified with the natural map:
$$\Hom_Z(Z,V)\to \Hom_{Z\times_S S_0}(Z\times_S S_0,V\times_S S_0)$$
where $V=X\times_Y Z$ is \'etale over $Z$.
Since $(Z,Z\times_S S_0)$ is a henselian pair
(Corollary~\ref{cor:hens_pair_prop}), the result follows.
\end{proof}
\begin{rem}
Note that \cite[XII.6.5(i)]{SGA4} and \cite[XII.5.1(i)]{SGA4} are quite
elementary. The first result follows from the following facts: (i) every sheaf
is a direct limit of constructible sheaves~\cite[IX.2.7.2]{SGA4}, and (ii)
every constructible sheaf embeds in a product of push-forwards of constant
sheaves along finite morphisms~\cite[IX.2.14]{SGA4}. To prove the second result,
one reduces to the case where $S$ is henselian, then to $X=\Pr^n_S$ using a
suitable Chow lemma, and finally to $S$ noetherian and henselian using
approximation. Then one concludes using \cite[XII.5.1(i)]{SGA4} and the Stein
factorization~\cite[XII.5.8]{SGA4}.

For completeness, let us mention some generalizations of the results in this
section to algebraic spaces and stacks. The first result,
\cite[XII.6.5(i)]{SGA4}, is easily extended to noetherian stacks and to
quasi-compact and quasi-separated Deligne--Mumford stacks. The second result,
\cite[XII.5.1(i)]{SGA4}, then follows for noetherian stacks using the Chow
lemma~\cite[Thm.~1.1]{MR2183251} and for quasi-compact and quasi-separated
Deligne--Mumford stacks using the Chow
lemma~\cite[Thm.~B]{rydh-2009}. Proposition~\ref{prop:et_prop_geom_conn},
Corollary~\ref{cor:hens_pair_prop} and Corollary~\ref{cor:sections_proper_bc}
thus follow for such stacks. The noetherian assumption in
Corollary~\ref{cor:hens_pair_prop} can be removed for Deligne--Mumford stacks
using the fact that there exists Stein factorizations for proper morphisms of
non-noetherian stacks, although one gets an integral morphism instead of a
finite morphism. However, this is not a problem as integral morphisms can be
approximated by finite morphisms~\cite[Thm.~A]{rydh-2009}.
\end{rem}
\bibliography{bibtex_db/references}
\bibliographystyle{bibtex_db/dary}
\end{document}